\documentclass[11pt]{article}

\usepackage[T1]{fontenc}
\usepackage[utf8]{inputenc}
\usepackage{lmodern}
\usepackage[a4paper,margin=30mm]{geometry}
\usepackage{microtype}
\usepackage{amsmath,amssymb,amsthm,mathtools}
\usepackage{enumitem}
\usepackage{xcolor}
\usepackage[numbers]{natbib}
\usepackage[hidelinks]{hyperref}
\hypersetup{
  pdftitle={General Constructions of Permutation-Fair Dice},
  pdfauthor={Tomas Gavenciak; Bernhard Haeupler; Vaclav Rozhon},
  pdfsubject={Constructions, optimal individual sizes, and pointwise approximation for permutation-fair dice},
  pdfkeywords={permutation-fair dice, permutation-fair strings, Hahn polynomials, positive quadrature, Chebyshev-type quadrature, approximate fairness}
}

\newtheorem{theorem}{Theorem}[section]
\newtheorem{lemma}[theorem]{Lemma}
\newtheorem{proposition}[theorem]{Proposition}
\newtheorem{corollary}[theorem]{Corollary}
\newtheorem{claim}[theorem]{Claim}

\theoremstyle{definition}
\newtheorem{definition}[theorem]{Definition}
\theoremstyle{remark}
\newtheorem{remark}[theorem]{Remark}

\newcommand{\N}{\mathbb{N}}

\newcommand{\PiSet}{\mathop{\Pi}}
\newcommand{\fall}[2]{#1^{\underline{#2}}}

\newcommand{\concat}{\mathbin{\circ}}
\DeclareMathOperator{\lcm}{lcm}

\definecolor{dieRed}{HTML}{C8444F}
\definecolor{dieGreen}{HTML}{23865A}
\definecolor{dieBlue}{HTML}{326EB3}
\newcommand{\dietile}[2]{%
  \ensuremath{%
    \vcenter{\hbox{%
      \begingroup
      \setlength{\fboxsep}{0pt}%
      \colorbox{#1}{%
        \makebox[1.35em][c]{%
          \rule[-0.22em]{0pt}{1.35em}%
          \color{white}\sffamily\bfseries #2%
        }%
      }%
      \endgroup
    }}\mkern1mu%
  }%
}
\DeclareRobustCommand{\dieOne}{\dietile{dieRed}{1}}
\DeclareRobustCommand{\dieTwo}{\dietile{dieGreen}{2}}
\DeclareRobustCommand{\dieThree}{\dietile{dieBlue}{3}}

\title{General Constructions of Permutation-Fair Dice}
\author{%
  Tom\'a\v{s} Gaven\v{c}iak\\
  \small Centre for Theoretical Study\\[-1mm]
  \small Charles University, Prague\\[-1mm]
  \small\texttt{gavenciak@cts.cuni.cz}
  \and
  Bernhard Haeupler\\
  \small ETH Zurich\\[-1mm]
  \small\texttt{bernhard.haeupler@inf.ethz.ch}
  \and
  V\'aclav Rozho\v{n}\\
  \small Faculty of Mathematics and Physics\\[-1mm]
  \small Charles University, Prague\\[-1mm]
  \small\texttt{vaclavrozhon@gmail.com}}
\date{\today}

\begin{document}

\maketitle

\begin{abstract}
Suppose that each of \(n\) players rolls a different die and the players are
ordered by their rolls.  Can the dice use disjoint labels such that
ties are impossible, while all \(n!\) orders remain equally likely?

We give two general constructions of such dice and solve some related problems.
\end{abstract}

\section{Introduction}

There are \(n\) players, and player \(i\) is assigned a die with
\(k_i\geq1\) faces.  Across all dice, the faces receive the distinct labels
\(1,2,\ldots,\sum_i k_i\), each label appearing exactly once.  The players
roll independently and are ordered by their labels.  We call the dice
\emph{permutation-fair} if the resulting order is uniform on the \(n!\)
permutations.

The problem was raised by James Ernest and has been studied extensively for
small numbers of players; see, for example,
\cite{harshbarger-web,harshbarger2012,ford2023}.  Permutation-fair dice are
known to exist for every \(n\).  We present our elementary construction,
give a second and more efficient construction, and prove bounds on the
number of faces needed.

\subsection{Our results}\label{sec:results}

\paragraph{A three-die example.}
For three dice, consider the string
\[
  s=
  \dieTwo\dieOne\dieThree\dieOne
  \dieTwo\dieTwo\dieTwo\dieTwo
  \dieOne\dieThree\dieOne\dieTwo.
\]
The \(j\)-th box specifies which die owns the face labeled \(j\); red,
green, and blue represent dice \(1,2,3\), respectively.  Thus the string
represents the following three dice:
\[
  \dieOne:\{2,4,9,11\},\qquad
  \dieTwo:\{1,5,6,7,8,12\},\qquad
  \dieThree:\{3,10\}.
\]
A joint roll chooses one occurrence of each box color, and their
left-to-right order is the order of the rolled labels.  Each of the six
possible orders occurs exactly eight times as a subsequence of \(s\).  Since
there are \(4\cdot6\cdot2=48\) possible rolls, every order has probability
\(8/48=1/6\).

More generally, read all face labels in increasing order and record the owner
of each face.  This turns any collection of dice into a string \(s\), in which
symbol \(i\) occurs \(k_i\) times.  If \(C_\pi(s)\) denotes the number of
occurrences of a permutation \(\pi\) as a subsequence of \(s\), then
\[
  \Pr[\text{order}=\pi]
  =\frac{C_\pi(s)}{\prod_{i=1}^n k_i}.
\]
Consequently, the dice are permutation-fair exactly when all \(n!\)
permutations occur equally often in \(s\).  This correspondence replaces a
probabilistic question about dice by a combinatorial question about strings.

\paragraph{The elementary construction.}
Our first construction\footnote{We first presented this construction in our
YouTube video~\cite{polylog-video}, later cited by Purcell~\cite{purcell}.}
starts from \(12\cdots n\) and repeatedly
concatenates all \(n!\) relabelings of the string obtained so far.  For
\(n=3\), start with \(u_1\) and concatenate its six relabelings to obtain
\(u_2\):
\[
  \begin{aligned}
  u_1&=\dieOne\dieTwo\dieThree,\\[.4em]
  u_2&=
  \dieOne\dieTwo\dieThree\mid
  \dieOne\dieThree\dieTwo\mid
  \dieTwo\dieOne\dieThree\mid
  \dieTwo\dieThree\dieOne\mid
  \dieThree\dieOne\dieTwo\mid
  \dieThree\dieTwo\dieOne.
  \end{aligned}
\]
The string \(u_1\) has equal counts of all one-symbol patterns, and \(u_2\)
has equal counts of all two-symbol patterns.  Next, fix any
ordering \(\tau_1,\ldots,\tau_6\) of the six permutations of the symbols.
Write \(u_2^\tau\) for the string obtained by replacing each symbol \(i\)
by \(\tau(i)\), and let \(\concat\) denote concatenation.  Concatenating the
six relabelings in that order gives
\[
  u_3=u_2^{\tau_1}\concat u_2^{\tau_2}\concat\cdots\concat u_2^{\tau_6}.
\]
This string has equal counts of all three-symbol patterns, and hence it is
permutation-fair.  In general, each round enforces symmetry for patterns one
symbol longer, so \(n-1\) rounds produce a permutation-fair string.

\begin{theorem}[Elementary construction
  (informal version of Theorem~\ref{thm:first-size})]
For every \(n\), there are \(n\) equal-sided permutation-fair dice with
\(2^{O(n^2\log n)}\) faces each.
\end{theorem}

Using all \(n!\) relabelings at every step makes the strings grow quickly.
The second construction gives smaller dice.

\paragraph{The Hahn-polynomial construction.}
The second construction adds one symbol at a time:
\[
  v_1=\dieOne,
  \qquad
  v_2=
  \dieTwo\mid
  v_1\mid
  \dieTwo\dieTwo\dieTwo\dieTwo\mid
  v_1\mid
  \dieTwo,
\]
\[
  v_3=
  \dieThree\dieThree\mid v_2\mid
  \dieThree\dieThree\dieThree\dieThree
  \dieThree\dieThree\dieThree\dieThree
  \mid v_2\mid
  \dieThree\dieThree.
\]
Each string is obtained by placing copies of the previous string between
runs of the new symbol.  The final string \(v_3\) has \(4,12,12\)
occurrences of symbols \(1,2,3\), respectively.  Although the dice have
different numbers of faces, each full permutation occurs exactly \(96\)
times, so the string is permutation-fair.

The run lengths are dictated by a quadrature condition: after normalization,
they must integrate the Bernstein polynomials~\cite{lorentz1986} of the
current degree exactly.  Above, both \((1,4,1)/6\) and \((2,8,2)/12\) are
Simpson's rule.  Using Hahn polynomials and estimates of Wilson and
Zaremba~\cite{wilson-hahn,zaremba1975}, we find suitable positive rational
rules with small common denominators for every \(n\), which gives a sharper
asymptotic bound.

\begin{theorem}[Hahn construction
  (informal version of Theorem~\ref{thm:main})]
For every \(n\), there are \(n\) equal-sided permutation-fair dice with
\(2^{O(n\log n)}\) faces each.
\end{theorem}

\paragraph{Lower bounds.}
Restricting a fair string to any \(p\) symbols, where \(p\) is prime, forces
\(p\) to divide the product of their multiplicities.  Applying this
observation simultaneously for many primes gives the following result.

\begin{theorem}[Lower bound
  (informal version of Theorem~\ref{thm:many-large})]
Every permutation-fair collection of \(n\) dice contains a die with
\(2^{\Omega(n)}\) faces.
\end{theorem}

Thus, if \(f(n)\) is the smallest possible value of the largest die in a
permutation-fair collection, then
\[
  2^{\Omega(n)}\leq f(n)\leq2^{O(n\log n)}.
\]
The precise lower bound is stronger: for every fixed
\(0<\varepsilon<1\), it forces \(2^{\Omega(\varepsilon n)}\) faces on at
least \((1-\varepsilon)n\) dice.

\paragraph{The smallest individual die.}
The exponential bound applies to most dice, but a prescribed die can be much
smaller.  Let \(g(n)\) be the minimum possible number of faces of one die
in a permutation-fair family of \(n\) dice.

\begin{theorem}[Optimal individual size
  (informal version of Theorem~\ref{thm:individual-size})]
We have \(g(n)=\Theta(n^2)\).
\end{theorem}

\paragraph{Weaker and approximate notions.}
Sections~\ref{sec:weaker} and~\ref{sec:approximate} contain three further results.
Dice are \emph{go-first fair} if every player is equally likely to roll the
smallest label, and \emph{place fair} if every player is equally likely to
occupy each rank.  Here \emph{full-set} means that fairness is required when
all \(n\) dice are rolled, without imposing it on every subset of the dice.

\begin{enumerate}[label=\textup{(\roman*)},leftmargin=*,itemsep=.45em]
  \item \textbf{Go-first fairness
    (informal version of Theorem~\ref{thm:go-first-optimum}).}
    Among full-set go-first-fair strings, allowing different multiplicities,
    the minimum possible total length is exactly \(n(n+1)/2\).

  \item \textbf{Place fairness
    (informal version of Theorem~\ref{thm:place-prouhet}).}
    For every \(n\geq2\), there are \(n\) equal-sided, full-set place-fair
    dice with \(2n^{n-2}\) faces each.

  \item \textbf{Approximate permutation fairness
    (informal version of Theorem~\ref{thm:approximate}).}
    For every \(n\geq3\) and \(0<\varepsilon\leq1\), there are equal-sided
    dice with
    \[
      O\!\left(n^{13/10}(\log n)^{1/5}\varepsilon^{-2/5}\right)
    \]
    faces each that give every permutation probability within relative
    error \(\varepsilon\) of \(1/n!\).
\end{enumerate}

\paragraph{Roadmap.}
Section~\ref{sec:prelim} develops the string representation and the elementary
symmetrization construction.  Section~\ref{sec:second} gives the sharper
Hahn-polynomial construction.  Section~\ref{sec:lower} proves the lower
bounds and determines the minimum size of one prescribed die.
Section~\ref{sec:weaker} treats go-first and place fairness, and
Section~\ref{sec:approximate} treats approximate permutation fairness.
The technical estimates for the approximate construction are proved in
Appendix~\ref{app:approximate-proof}.
All construction and bound results in this article have been formalized in
Lean~4, modulo classical inputs.\footnote{These are the existence of
Gaussian and Radau quadrature rules, the Gilboa--Peled equal-weight
quadrature theorem, and Bonami's hypercontractive inequality.  They are
stated as explicit hypotheses.}
The source code and documentation of the formalization are in the
accompanying repository~\cite{fairdice-lean}.

\subsection{Related work}

The dice problem can be viewed as a rigid finite-precision version of the
random-key method for sampling a permutation.  Draw \(n\) independent uniform
keys in \([0,1]\) and sort them.  Ties have probability zero, and symmetry
makes every order equally likely; see the discussion of uniform order
statistics in~\cite[Sec.~I.4.3]{devroye1986}.  With finite keys, one normally tolerates a small collision probability
or resamples when a collision occurs.  Neither option is allowed here: every
player makes one independent uniform draw, the supports are pairwise
disjoint, and the induced order must be exactly uniform.  This rigidity has
a substantial cost: by our lower bound, a constant fraction of the players
need supports of size \(2^{\Omega(n)}\), for a total of \(\Omega(n^2)\) bits
of entropy per draw, whereas standard approaches need only
\(O(\log n!)=O(n\log n)\) bits.

Permutation-fair, place-fair, and go-first-fair dice have been investigated
for small \(n\) in~\cite{harshbarger-web,harshbarger2012,ford2023,purcell}.
The insertion framework underlying our second construction originates in
this literature; an early online account attributes an inductive version to
Paul Vanetti~\cite{knoll2012}.  Ford et al. give five equal 60-sided
hereditarily place-fair dice, unequal-sided permutation-fair examples, and
five equal 180-sided permutation-fair dice~\cite{ford2023}.  Purcell develops
the symmetrization construction further and uses it to construct five equal
120-sided permutation-fair dice~\cite{purcell}.  A computational construction
due to Paul Meyer gives five equal 60-sided permutation-fair
dice~\cite{meyer60}.  We are not aware of a
previous treatment of the exact full-set, unequal-multiplicity go-first
optimum proved in Section~\ref{sec:weaker}.

A weaker string problem asks only that every permutation occur at least once
as a subsequence.  This has a long history
\cite{chvatal1972,newey1973,adleman1974,koutas1975,galbiati1976,kleitman1976,
zalinescu2011,radomirovic2012}; see Engen and Vatter~\cite{engen2020} for a
survey.  A superficially similar but different problem asks for an ordinary
permutation in which all order-isomorphic patterns of a fixed length occur
equally often.  Beniamini, Lavee, and Linial settle that problem, including
nonexistence from pattern length four onward~\cite{beniamini2025}; their host
permutation has no repeated letters, so that obstruction does not apply here.

There is also an established literature on \emph{fair words}.  In \v{C}ern\'y's
terminology~\cite{cerny2007,cerny2009}, a word is fair if
\(C_{ab}=C_{ba}\) for every pair of distinct letters; the binary case goes
back at least to Prodinger~\cite{prodinger1979}.  Among words in which every
letter occurs equally often, this is exactly our \((2,n)\)-fairness
(Definition~\ref{def:pattern-fairness}):
every ordered pair of distinct letters has the same subsequence count.
\v{C}ern\'y also studied
reversal symmetry of longer scattered subwords~\cite{cerny2008}; see
Salomaa~\cite{salomaa2007} for related subword-balance questions and
Richomme~\cite{richomme2026} for a recent account.  To avoid confusing this
pairwise notion with equality of all full-permutation counts, we use
\emph{permutation-fair string} for the latter in prose while retaining the
compact notation \(n\)-fair and \((k,n)\)-fair.

The same pairwise condition appears in work on tied and balanced
nontransitive dice, often through the same word encoding
\cite{kronenthal-traldi2009,cooley-et-al2014,schaefer-schweig2017,
hur-kim2021,kim-et-al2024}.  Those papers prescribe pairwise win
probabilities, rather than the joint distribution of the full order, and are
therefore adjacent to but weaker than the permutation-fair problem.

Our second construction also draws on classical equidistant quadrature and
the Hahn polynomials.  Wilson gave the discrete least-squares projection
formula for the quadrature weights and studied equidistant quadrature
\cite{wilson-algorithm,wilson-least-squares,wilson-quadrature}; his Hahn
estimates control the continuous integrals of the basis
polynomials~\cite{wilson-hahn}.  Huybrechs revisited Wilson's stable
high-order rules on equidistant points~\cite{huybrechs2009}, and later
least-squares quadrature work treats scattered nodes and other weight
functions~\cite{glaubitz2020,cappellazzo-et-al2023}.  Classical positive
quadrature goes back at least to Fej\'er and Shohat
\cite{fejer1933,shohat1937}.  For
positivity we use Zaremba's sharper bound at the integer nodes
\cite{zaremba1975}; Dette later obtained bounds on the whole interval
\cite{dette1995}.  Our use of these ingredients is arithmetic: we give an
explicit common-denominator estimate and show how one rule can be reused in
every insertion step.

More broadly, existence theorems for unweighted designs on connected spaces
go back to Seymour and Zaslavsky~\cite{seymour-zaslavsky1984}, with modern
quantitative results due to Kane~\cite{kane2015}.  Chebyshev-type interval
quadrature for the uniform measure has quadratic-order real-node bounds going
back to Bernstein, with later sharp formulations by Kuijlaars and by Gilboa
and Peled~\cite{bernstein1937,bernstein-positive1937,kuijlaars1993,kuijlaars1995,gilboa2017}; see also Peled's
general-measure bounds~\cite{peled2010}.  Imposing rational nodes leads to the
rational-design problem of Cui, Xia, and Xiang~\cite{cui2019}, a related
arithmetic variant of the quadrature problem whose solution builds in part
on Hilbert--Kamke methods~\cite{arkhipov1985}.  Our construction
uses real quadrature nodes and corrects their rational approximation by
changing the distributions of the other dice.

Approximate dice connect to classical shuffling laws. The periodic word
\((12\cdots n)^m\) realizes an inverse \(m\)-riffle shuffle: each symbol
independently chooses one of \(m\) levels uniformly, and the symbols are
ordered by level, breaking ties in increasing symbol order.  Its probabilities
are described by descent enumeration~\cite{stanley1972,bayer-diaconis1992}.
Repeating the palindrome \(12\cdots n\,n\cdots21\) realizes a shelf
shuffle, with the same level assignment but alternating increasing and
decreasing symbol order within the \(2m\) levels.
Diaconis, Fulman, and Holmes~\cite{diaconis-fulman-holmes2013}
analyze its maximal relative error and establish the \(n^{3/2}\) scale for
fixed accuracy. We independently relabel the palindrome blocks and bound
the resulting error using multilinear hypercontractivity
\cite{bonami1970,odonnell2014} and a multinomial collision estimate.

\section{Permutation-fair strings and elementary construction}\label{sec:prelim}

\subsection{Fair strings and their properties}

Write \([n]=\{1,\ldots,n\}\).  A \emph{pattern} is a string with no repeated
letter.  Let \(\PiSet_{n,k}\) be the set of all length-\(k\) patterns over
\([n]\), and let \(\PiSet_n=\PiSet_{n,n}\).  For a pattern \(\pi\) and a string
\(s\), write \(C_\pi(s)\) for the number of occurrences of \(\pi\) as a
subsequence of \(s\).  We put \(C_\varnothing(s)=1\).

\begin{definition}[String representation]
Let \(D_1,\ldots,D_n\) be dice whose face labels partition \([K]\).  Their
\emph{string representation} is the string \(s\in[n]^K\) defined by
\(s[a]=i\) when label \(a\) is a face of \(D_i\).
\end{definition}

\begin{definition}[Permutation-fair string]
A string \(s\) over \([n]\) in which every letter occurs at least once is
\emph{\(n\)-fair}, or \emph{permutation-fair}, if \(C_\pi(s)\) has the same
value for every \(\pi\in\PiSet_n\).
\end{definition}

\begin{lemma}\label{lem:dice-string}
A set of dice is permutation-fair if and only if its string representation is
permutation-fair.
\end{lemma}

\begin{proof}
Choose one face from each die.  Sorting the chosen labels produces an
occurrence of a unique pattern \(\pi\in\PiSet_n\) in the string
representation.  Conversely, every such occurrence selects exactly one face
from every die.  Since all \(\prod_i k_i\) face selections have the same
probability, the orders are equiprobable exactly when their subsequence counts
are equal.
\end{proof}

Relabeling the alphabet preserves fairness.  Restricting to fewer letters
does as well: after fixing one face of every deleted die, summing over the
positions of those faces shows that all orders of the retained letters have
the same count.  We record this together with another useful operation.

\begin{claim}\label{clm:restriction-blowup}
Let \(s\) be \(n\)-fair.
\begin{enumerate}[label=(\roman*)]
  \item Relabeling the letters of \(s\) by a permutation of \([n]\) preserves
        fairness, and restricting \(s\) to any subset of its letters produces
        a permutation-fair string on that subset.
  \item Replacing every occurrence of a fixed letter \(i\) by a run of \(q\)
        copies of \(i\), where \(q\in\N\), preserves fairness.
\end{enumerate}
\end{claim}

\begin{proof}
Relabeling merely permutes the pattern counts.  The restriction statement
follows by the preceding summation argument.  For (ii), every occurrence of a
full pattern in \(s\) has exactly \(q\) lifts to the enlarged string, one for
each copy of its unique occurrence of \(i\).
\end{proof}

In particular, arbitrary permutation-fair dice may be made equal-sided, although this
generic operation is much more expensive than the equal-sided construction of
Section~\ref{sec:second}.

\begin{corollary}\label{cor:equalize}
If \(s\) is \(n\)-fair, there is an \(n\)-fair string \(s'\) in which every
letter has the same multiplicity and
\[
  |s'|\leq |s|^n.
\]
\end{corollary}

\begin{proof}
If letter \(i\) occurs \(c_i\) times, blow it up by
\(\prod_{j\neq i}c_j\).  Every letter then occurs \(\prod_jc_j\) times, so
\[
  |s'|=n\prod_jc_j
  \leq n\left(\frac{|s|}{n}\right)^n
  \leq |s|^n.
\]
\end{proof}

Full fairness determines all partial-pattern counts.

\begin{lemma}[Partial-pattern count]\label{lem:partial}
Let \(s\) be \(n\)-fair, let \(C\) be the common count of a full pattern, and
let \(c_i=C_i(s)\).  If \(\pi\in\PiSet_{n,k}\) uses the set \(S\subseteq[n]\),
then
\begin{equation}\label{eq:partial-count}
  C_\pi(s)
  =\frac{(n!/k!)C}{\prod_{\ell\notin S}c_\ell}.
\end{equation}
\end{lemma}

\begin{proof}
Count pairs consisting of an occurrence of \(\pi\) and one chosen occurrence
of each omitted letter.  Directly, their number is
\(C_\pi(s)\prod_{\ell\notin S}c_\ell\).  Sorting all chosen positions gives a
full pattern whose restriction to \(S\) is \(\pi\).  There are \(n!/k!\) such
full patterns, each with count \(C\), proving~\eqref{eq:partial-count}.
\end{proof}

\begin{corollary}\label{cor:self-concat}
If \(s\) is \(n\)-fair, then \(s\concat s\), and hence every positive power
\(s^m\), is \(n\)-fair.
\end{corollary}

\begin{proof}
For \(\pi\in\PiSet_n\), split an occurrence according to the number \(j\) of
letters chosen from the first copy:
\[
  C_\pi(s\concat s)
  =\sum_{j=0}^n C_{\pi[1:j]}(s)C_{\pi[j+1:n]}(s).
\]
For fixed \(j\), applying Lemma~\ref{lem:partial} to the two complementary
patterns makes the product independent of \(\pi\): its denominator contains
each \(c_i\) exactly once.  The result follows, and iteration gives \(s^m\).
\end{proof}

For the first construction, it is convenient to ask for fairness at every
pattern length.

\begin{definition}\label{def:pattern-fairness}
A string \(s\) over \([n]\) is \emph{\((k,n)\)-fair} if \(C_\pi(s)\) is the
same for all \(\pi\in\PiSet_{n,k}\).  It is \emph{\((\leq k,n)\)-fair} if it
is \((j,n)\)-fair for every \(0\leq j\leq k\).
\end{definition}

The concatenation of two \((\leq k,n)\)-fair strings is again
\((\leq k,n)\)-fair: split an occurrence at the concatenation point and use
fairness of the two pieces at the corresponding lengths.  Equal-composition
permutation-fair strings automatically satisfy this stronger property.

\begin{lemma}\label{lem:equal-full-to-partial}
If \(s\) is \(n\)-fair and every letter occurs equally often, then \(s\) is
\((\leq n,n)\)-fair.
\end{lemma}

\begin{proof}
In~\eqref{eq:partial-count}, the denominator depends only on the number of
omitted letters, not on their identities.  Thus every pattern count depends
only on its length.
\end{proof}

\subsection{The first construction: symmetrization}\label{sec:first}

We first give an elementary construction.  Rather than adding letters one at
a time, it progressively enforces symmetry for longer patterns.

If \(\tau\in\PiSet_n\) is viewed as a permutation of \([n]\), let \(s^\tau\)
denote the string obtained from \(s\) by replacing every letter \(i\) by
\(\tau(i)\).

\begin{lemma}[Symmetrization; cf.\ \cite{purcell}]\label{lem:symmetrize}
Let \(k<n\).  From any \((\leq k,n)\)-fair string \(s\) of length \(\ell\), one
can construct a \((\leq k+1,n)\)-fair string of length \(\ell n!\).
\end{lemma}

\begin{proof}
List the \(n!\) permutations as \(\tau_1,\ldots,\tau_{n!}\), in any order, and
set
\[
  s'=s^{\tau_1}\concat s^{\tau_2}\concat\cdots\concat s^{\tau_{n!}}.
\]
Fix \(t\leq k+1\), a pattern \(\pi\in\PiSet_{n,t}\), and an occupancy vector
\((q_1,\ldots,q_{n!})\) specifying how many consecutive letters of \(\pi\) are
selected from each block.  If every \(q_a\leq k\), the contribution of this
occupancy vector is a product of counts of patterns of lengths \(q_a\) in
relabelings of \(s\).  Since \(s\) is \((\leq k,n)\)-fair, that product
depends only on the vector, not on \(\pi\).

The only remaining case is \(t=k+1\) and all \(k+1\) selected letters lie in
one block.  Summed over these \(n!\) possible blocks, the contribution is
\[
  \sum_{\tau\in\PiSet_n} C_\pi(s^\tau).
\]
As \(\tau\) ranges over all permutations, \(\tau^{-1}(\pi)\) ranges uniformly
over the length-\((k+1)\) patterns: each such pattern occurs
\((n-k-1)!\) times.  This sum is therefore independent of \(\pi\).  Summing
over all occupancy vectors proves that \(s'\) is
\((\leq k+1,n)\)-fair.
\end{proof}

\begin{theorem}[Elementary symmetrization]\label{thm:first-size}
For every \(n\), there are \(n\) equal-sided permutation-fair dice with
\((n!)^{n-1}\) faces each.  Equivalently, there is a
\((\leq n,n)\)-fair string of length
\[
  n(n!)^{n-1}=2^{O(n^2\log n)}.
\]
\end{theorem}

\begin{proof}
The string \(s_1=12\cdots n\) is \((\leq1,n)\)-fair.  Apply
Lemma~\ref{lem:symmetrize} \(n-1\) times.  Each application multiplies the
length by \(n!\).
\end{proof}

This proves the qualitative existence of permutation-fair dice for every
\(n\).  The next construction improves the exponent and keeps every die the
same size throughout.

\section{The second construction: Hahn quadrature}\label{sec:second}

We now add one letter at a time.  Fix the final target \(n\), and let \(s\) be
an equal-composition \(r\)-fair string, where \(r<n\).  Given \(N\) copies of
\(s\) and nonnegative integers \(k_0,\ldots,k_N\), form
\begin{equation}\label{eq:interlace}
 T=(r+1)^{k_0}\concat s\concat(r+1)^{k_1}\concat s\concat\cdots
   \concat s\concat(r+1)^{k_N}.
\end{equation}
There are \(N+1\) gaps, numbered \(0,\ldots,N\), around \(N\) copies of \(s\).

\subsection{The insertion step is a quadrature problem}

Let every old symbol occur \(F\) times in \(s\), and put
\(K=\sum_i k_i\) and \(w_i=k_i/K\).  The Bernstein
polynomials~\cite{lorentz1986}
\[
  b_{r,j}(x)=\binom rj x^j(1-x)^{r-j},\qquad 0\leq j\leq r,
\]
satisfy
\begin{equation}\label{eq:bernstein-integral}
  \int_0^1 b_{r,j}(x)\,dx=\frac1{r+1}.
\end{equation}

\begin{lemma}[Insertion criterion]\label{lem:insertion}
If
\begin{equation}\label{eq:insertion-moments}
  \sum_{i=0}^N w_i b_{r,j}(i/N)=\frac1{r+1}
  \qquad(0\leq j\leq r),
\end{equation}
then the string \(T\) in~\eqref{eq:interlace} is \((r+1)\)-fair.
\end{lemma}

\begin{proof}
By Lemma~\ref{lem:equal-full-to-partial}, \(s\) is
\((\leq r,r)\)-fair; so is each power \(s^a\).  Consequently, for every
pattern \(\pi\in\PiSet_{r,j}\) of old symbols,
\[
  C_\pi(s^i)=\frac{(iF)^j}{j!},
\]
because all \(j!\) orders partition the \((iF)^j\) choices of one occurrence
of each symbol.  The analogous suffix count is
\(((N-i)F)^{r-j}/(r-j)!\).

Fix a full pattern in which the new symbol has rank \(j+1\).  If its chosen
occurrence of \(r+1\) lies in gap \(i\), the number of choices for the old
symbols is therefore
\[
  \frac{(iF)^j}{j!}\,
  \frac{((N-i)F)^{r-j}}{(r-j)!}.
\]
After summing over all \(k_i\) occurrences in every gap, its total count is
\begin{align*}
  \sum_{i=0}^N k_i
  \frac{(iF)^j((N-i)F)^{r-j}}{j!(r-j)!}
  & =\frac{KF^rN^r}{r!}
    \sum_{i=0}^N w_i b_{r,j}(i/N)\\
  & =\frac{KF^rN^r}{(r+1)!},
\end{align*}
which is independent of both \(j\) and the order of the old symbols.
\end{proof}

Thus it suffices to find positive rational weights \(w_i\) that integrate all
polynomials of degree at most \(n\) on \([0,1]\).  The same rule will then
satisfy~\eqref{eq:insertion-moments} at every stage \(r<n\).

\begin{remark}[Why positivity needs care]\label{rem:positivity-gap}
A Riemann sum is close to the integral, but correcting its moments at only
\(n+1\) selected grid points does not automatically preserve nonnegativity.
For example, in degree \(2\), correcting the uniform weights \(1/m\) at
\(0,1/3,2/3\) gives at \(1/3\) the weight
\[
  -\frac1{2m}+\frac3{2m^2},
\]
which is negative for every \(m>3\).  The Hahn-basis projection below spreads the
correction across the entire grid and gives a uniform positivity estimate.
\end{remark}

\subsection{Hahn polynomials on an equally spaced grid}

Set
\begin{equation}\label{eq:N-choice}
  N=n(n+2),
  \qquad N+1=(n+1)^2.
\end{equation}
For \(0\leq k\leq n\), define the discrete Chebyshev (Hahn) polynomial
\begin{equation}\label{eq:Q-hypergeom}
  Q_k(x)=Q_k(x;0,0,N+1)
  ={}_3F_2\!\left(\begin{matrix}-k,-x,k+1\\1,-N\end{matrix};1\right).
\end{equation}
In Wilson's normalization~\cite{wilson-hahn}, \(Q_k(0)=1\), and
\begin{equation}\label{eq:Q-expansion}
  Q_k(x)=\sum_{a=0}^k(-1)^a
  \binom ka\binom{k+a}{a}
  \frac{\fall{x}{a}}{\fall{N}{a}}.
\end{equation}
Here \(\fall{x}{a}=x(x-1)\cdots(x-a+1)\), with \(\fall{x}{0}=1\),
is the falling factorial.
The polynomials \(Q_0,\ldots,Q_n\) are orthogonal for the uniform measure on
\(\{0,1,\ldots,N\}\).  Define the unnormalized squared norm and the continuous
integral
\[
  H_k=\sum_{i=0}^NQ_k(i)^2,
  \qquad
  I_k=\int_0^NQ_k(x)\,dx.
\]
The standard Hahn norm formula
(see the primary formulas in~\cite{karlin-mcgregor1961,dette1995}, as well as
\cite[Sec.~9.5, Eq.~(9.5.2)]{koekoek2010} and
\cite[Table~18.19.1]{dlmf}, translated from their largest-grid-index
convention to Wilson's number-of-grid-points convention) specializes to
\begin{equation}\label{eq:H-norm}
  H_k
  =\frac{N+1}{2k+1}
   \frac{(N+2)(N+3)\cdots(N+k+1)}{N(N-1)\cdots(N-k+1)}
  =\frac{(N+k+1)!(N-k)!}{(2k+1)(N!)^2}.
\end{equation}
In particular,
\begin{equation}\label{eq:H-lower}
  H_k\geq\frac{N+1}{2k+1}.
\end{equation}

\begin{lemma}[Zaremba's nodal bound]\label{lem:Q-supnorm}
For every \(0\leq k\leq n\) and every integer \(i\in\{0,\ldots,N\}\),
\begin{equation}\label{eq:Q-supnorm}
  |Q_k(i)|\leq1.
\end{equation}
\end{lemma}

\begin{proof}
In the largest-grid-index convention used by Zaremba and Dette, our
\(Q_k(x;0,0,N+1)\) is \(Q_k(x,0,0,N)\).  Zaremba's nodal estimate
\cite{zaremba1975} (see also~\cite[Remark~3.3]{dette1995}) gives
\(|Q_k(i)|\leq1\) at every integer \(0\leq i\leq N\) whenever
\[
  k(k+1)\leq N.
\]
Here
\[
  k(k+1)\leq n(n+1)<n(n+2)=N.
\]
The case \(k=0\) is immediate.
\end{proof}

Wilson's aim was to quantify how closely Hahn polynomials on a uniform grid
imitate their Jacobi limits~\cite{wilson-hahn}.  He developed an expansion in
inverse powers of the grid size and used it to compare discrete and continuous
least-squares expansions.  In the Legendre case he isolated the integrals
\(I_k\): they measure the failure of the nonconstant discrete orthogonal basis
to have zero continuous mean, and they also enter quadrature formulas on
equidistant nodes.  In our argument they are precisely the coefficients that
must be controlled in the projection weights below.

\begin{lemma}[Wilson's integral estimate]\label{lem:wilson-integral}
Let \(N\geq1\), let \(0\leq k\leq N\), and let \(Q_k\) and \(I_k\) be as
above.  If \(k\) is odd, then \(I_k=0\).  If \(k\geq2\) is even and
\(N+1\geq(k+1)^2\), then
\begin{equation}\label{eq:I-wilson}
  0<-I_k\leq\frac{N+1}{N}
  \left(1+\frac{(k-1)(k+2)}{4(N-1)}\right).
\end{equation}
Consequently, for the choice \(N=n(n+2)\),
\begin{equation}\label{eq:I-bound}
  |I_k|<\frac32\qquad(1\leq k\leq n).
\end{equation}
\end{lemma}

\begin{proof}
The reflection identity \(Q_k(N-x)=(-1)^kQ_k(x)\) gives
\(I_k=(-1)^kI_k\), and hence the odd case.  Wilson writes our \(I_k\) as
\(I_k(N+1)\).  For even \(k\), his lower estimate shows that \(-I_k>0\),
while his non-asymptotic upper estimate is exactly~\eqref{eq:I-wilson}; see
\cite[Sec.~3, Eqs.~(3.1) and (3.14)]{wilson-hahn}.

For the consequence, when \(n=1\) the only index is \(k=1\), which is odd.
If \(n\geq2\), then
\(N+1=(n+1)^2\geq(k+1)^2\), and
\[
  \frac{N+1}{N}\leq\frac98,
  \qquad
  \frac{(k-1)(k+2)}{4(N-1)}
  \leq\frac{(n-1)(n+2)}{4(N-1)}<\frac14.
\]
Thus \(|I_k|<45/32<3/2\) for even \(k\), while the odd integrals vanish.
\end{proof}

Wilson's estimates apply to finite grids.  Using
a recurrence due to Levit, he rewrites \(I_k\) as a weighted integral of a
neighboring Hahn polynomial; its zeros yield convexity and endpoint bounds,
and a coefficient-ratio argument turns those bounds into a geometric-series
estimate.  This quadratic relation between grid size and degree allows
\(k\) to grow with \(n\).

\subsection{Positive exact weights}

The following formula is the discrete least-squares projection formula for
quadrature weights used by Wilson~\cite{wilson-least-squares}: it expands the
integral functional in the orthogonal basis \(Q_0,\ldots,Q_n\).
For \(0\leq i\leq N\), put
\begin{equation}\label{eq:W-def}
  W_i=\sum_{k=0}^n\frac{I_k}{H_k}Q_k(i).
\end{equation}

\begin{proposition}[Exactness and positivity]\label{prop:quadrature}
The weights \(W_0,\ldots,W_N\) are strictly positive and satisfy
\begin{equation}\label{eq:quadrature-N}
  \sum_{i=0}^NW_ip(i)=\int_0^Np(x)\,dx
  \qquad\text{for every polynomial \(p\) with \(\deg p\leq n\)}.
\end{equation}
Moreover, \(\sum_iW_i=N\).
\end{proposition}

\begin{proof}
For \(0\leq\ell\leq n\), discrete orthogonality gives
\[
  \sum_{i=0}^NW_iQ_\ell(i)
  =\sum_{k=0}^n\frac{I_k}{H_k}
    \sum_{i=0}^NQ_k(i)Q_\ell(i)
  =I_\ell.
\]
The \(Q_\ell\) form a basis for the polynomials of degree at most \(n\), which
proves exactness.  Taking \(p=1\) yields \(\sum_iW_i=N\).

It remains to prove positivity.  Since \(Q_0=1\), \(I_0=N\), and
\(H_0=N+1\), the \(k=0\) term in~\eqref{eq:W-def} is \(N/(N+1)\).
If \(n=1\), then \(I_1=0\), so every \(W_i=3/4>0\).  Suppose now that
\(n\geq2\), and put \(m=\lfloor n/2\rfloor\).  Since all odd \(I_k\) vanish,
Lemma~\ref{lem:Q-supnorm}, \eqref{eq:I-bound}, and~\eqref{eq:H-lower}
give, uniformly in \(0\leq i\leq N\),
\begin{align*}
  \left|W_i-\frac{N}{N+1}\right|
  &\leq \sum_{\substack{2\leq k\leq n\\ k\ \mathrm{even}}}
       \frac{|I_k|\,|Q_k(i)|}{H_k}\\
  &<\frac{3}{2(N+1)}\sum_{j=1}^{m}(4j+1)\\
  &=\frac{3(2m^2+3m)}{2(N+1)}\\
  &\leq\frac{3n(n+3)}{4(N+1)}.
\end{align*}
Consequently,
\[
  W_i>\frac{N-3n(n+3)/4}{N+1}
      =\frac{n(n-1)}{4(N+1)}>0.
\]
\end{proof}

After scaling from \([0,N]\) to \([0,1]\), the normalized weights
\(w_i=W_i/N\) satisfy
\begin{equation}\label{eq:quadrature-unit}
  \sum_{i=0}^Nw_iq(i/N)=\int_0^1q(x)\,dx
  \qquad(\deg q\leq n).
\end{equation}

\subsection{A small common denominator}

Positivity alone is not enough: the gap lengths must be integers.  We now
bound a common denominator explicitly.  Let
\[
  A_k=\fall{N}{k}=N(N-1)\cdots(N-k+1),
  \qquad
  B_k=(N+2)(N+3)\cdots(N+k+1),
\]
with \(A_0=B_0=1\), and let
\[
  L=\lcm(1,2,\ldots,n+1).
\]
Formula~\eqref{eq:Q-expansion} shows that
\begin{equation}\label{eq:Q-denom}
  A_kQ_k(x)\in\mathbb Z[x],
  \qquad Q_k(i)\in A_k^{-1}\mathbb Z\quad(i\in\mathbb Z).
\end{equation}
Integrating the integer-coefficient polynomial \(A_kQ_k(x)\) term by term
between integer endpoints introduces only denominators \(1,\ldots,k+1\).
Thus
\begin{equation}\label{eq:I-denom}
  I_k\in(A_kL)^{-1}\mathbb Z.
\end{equation}
On the other hand,~\eqref{eq:H-norm} says
\[
  H_k=\frac{(N+1)B_k}{(2k+1)A_k}.
\]
Combining these facts gives
\begin{equation}\label{eq:term-denom}
  \frac{I_k}{H_k}Q_k(i)
  \in\bigl((N+1)LA_kB_k\bigr)^{-1}\mathbb Z.
\end{equation}
Since \(A_k\mid A_n\) and \(B_k\mid B_n\), the single integer
\begin{equation}\label{eq:M-def}
  M=(N+1)LA_nB_n
\end{equation}
clears all the weights:
\begin{equation}\label{eq:MW-integer}
  MW_i\in\mathbb Z_{>0}\qquad(0\leq i\leq N).
\end{equation}

Finally, \(L\leq(n+1)!\), \(A_n\leq N^n\), and
\(B_n\leq(2N)^n\).  Therefore
\begin{equation}\label{eq:M-bound}
  M\leq2^{n+1}(n+1)^{n+1}N^{2n+1}=2^{O(n\log n)}.
\end{equation}

\subsection{Reusing one rule at every induction level}

\begin{theorem}[Hahn construction]\label{thm:main}
For every \(n\geq1\), there are \(n\) equal-sided permutation-fair dice
with \(F_n\) faces each, where, for \(N=n(n+2)\),
\begin{equation}\label{eq:F-explicit}
  \boxed{\begin{aligned}
  F_n={}&(N+1)\lcm(1,\ldots,n+1)\fall{N}{n}\\
        &{}\cdot(N+2)(N+3)\cdots(N+n+1)N^{n-1}.
  \end{aligned}}
\end{equation}
In particular,
\[
  F_n=2^{O(n\log n)}.
\]
The corresponding string is \((\leq n,n)\)-fair.
\end{theorem}

\begin{proof}
Use the positive weights and the denominator \(M\) constructed above, and
start with
\[
  s_1=1^M.
\]
Inductively suppose \(s_r\) is \(r\)-fair and every symbol occurs
\[
  F_r=MN^{r-1}
\]
times.  Around \(N\) copies of \(s_r\), put in gap \(i\) a run of the new
symbol of length
\begin{equation}\label{eq:k-i}
  k_i=F_rW_i.
\end{equation}
These numbers are positive integers because \(F_r\) is divisible by \(M\).
By Proposition~\ref{prop:quadrature},
\[
  \sum_{i=0}^Nk_i=F_r\sum_{i=0}^NW_i=NF_r.
\]
Each old symbol also occurs \(NF_r\) times in the \(N\) copies.  Thus every
symbol in the new string occurs
\[
  F_{r+1}=NF_r=MN^r
\]
times.  The normalized gap weights are \(k_i/(NF_r)=W_i/N\), and
\eqref{eq:quadrature-unit} is exact for all degree-\(r\) Bernstein
polynomials.  Lemma~\ref{lem:insertion} therefore proves that the new string
is \((r+1)\)-fair.

After \(n-1\) insertion steps, every symbol occurs \(MN^{n-1}\) times.
Substituting the definition~\eqref{eq:M-def} of \(M\) gives
\eqref{eq:F-explicit}.  Using \(N+1\leq2N\) and the estimates above,
\[
  F_n\leq2^{n+1}(n+1)^{n+1}N^{3n}=2^{O(n\log n)}.
\]
The total string length is \(nF_n=2^{O(n\log n)}\).  Since all letter
multiplicities are equal, Lemma~\ref{lem:equal-full-to-partial} also makes the
final string \((\leq n,n)\)-fair.
\end{proof}

The rule in Theorem~\ref{thm:main} uses exactly \(N+1=(n+1)^2\) grid
points.  The quadratic grid changes the constants in the construction but
not its \(2^{O(n\log n)}\) asymptotic face bound.

\begin{remark}[A three-die check]
The strings \(v_2\) and \(v_3\) in Section~\ref{sec:results} use Simpson's
degree-three rule on \(0,1,2\) in place of the general-purpose rule:
\[
  (W_0,W_1,W_2)=\left(\frac13,\frac43,\frac13\right),
  \qquad N=2.
\]
Both gap vectors normalize to
\((1/6,2/3,1/6)=(W_0,W_1,W_2)/N\).
\end{remark}

\section{Sizes of permutation-fair dice}\label{sec:lower}

Divisibility across restrictions forces exponentially many faces on most
dice. We first prove that statement, then determine the minimum possible
size of one prescribed die.

\begin{theorem}[Many dice are large]\label{thm:many-large}
There is an absolute constant \(c>0\) such that, for every fixed
\(0<\varepsilon<1\) and all sufficiently large \(n\), every
permutation-fair collection with face counts \(k_1,\ldots,k_n\) satisfies
\[
  \bigl|\{i\in[n]: k_i\geq2^{c\varepsilon n}\}\bigr|
  \geq (1-\varepsilon)n.
\]
\end{theorem}

\subsection{Divisibility across restrictions}

The first argument is arithmetic.  Divisibility obstructions are classical
in the fair-dice literature; in the equal-sided setting, the corresponding
primorial condition is recorded by Purcell~\cite{purcell}.  Applying the same
idea after restriction to many different subsets gives a useful bound even
when the dice have unequal numbers of faces.

\begin{lemma}\label{lem:exp-lower}
Every \(n\)-fair string has length \(2^{\Omega(n)}\).
\end{lemma}

\begin{proof}
Let \(c_i=C_i(s)\).  Restrict \(s\) to a set \(S\) of \(p\) letters, where
\(p\leq n\) is prime.  The restricted string remains \(p\)-fair.  If \(C\) is
its common full-pattern count, then
\[
  \prod_{i\in S}c_i=p!C.
\]
Hence at least one \(c_i\), \(i\in S\), is divisible by \(p\).  Since this
holds for every \(p\)-subset, at most \(p-1\) of the \(n\) numbers \(c_i\) can
fail to be divisible by \(p\).  Consequently,
\[
  \prod_{i=1}^nc_i
  \geq\prod_{\substack{p\leq n/2\\p\ \mathrm{prime}}}p^{n/2}.
\]
The primorial estimate
\(\prod_{p\leq x}p=2^{\Theta(x)}\)~\cite{finch2003} now gives
\(\prod_i c_i=2^{\Omega(n^2)}\).  Therefore
\(\max_i c_i\geq2^{\Omega(n)}\), and the string is at least that long.
\end{proof}

\begin{proof}[Proof of Theorem~\ref{thm:many-large}]
Let \(r=\lfloor\varepsilon n\rfloor\), and restrict a fair string to its
\(r\) least frequent symbols.  By Lemma~\ref{lem:exp-lower}, the restricted
string has length at least \(2^{c_0r}\) for some absolute \(c_0>0\), so one
of these \(r\) symbols occurs at least \(2^{c_0r}/r\) times.  Each of the
other \(n-r\) symbols occurs at least as often.  For fixed \(\varepsilon\)
and sufficiently large \(n\),
\[
  \frac{2^{c_0r}}{r}\geq2^{(c_0/4)\varepsilon n}
  \qquad\text{and}\qquad
  n-r\geq(1-\varepsilon)n.
\]
This proves Theorem~\ref{thm:many-large} with \(c=c_0/4\).
\end{proof}

It is important not to overinterpret this argument.  For every subset
\(S\subseteq[n]\), restriction and permutation fairness imply the basic
divisibility condition
\[
  |S|!\ \bigm|\ \prod_{i\in S}c_i.
\]
These conditions alone are compatible with one prescribed die having a
single face: the multiplicity vector
\[
  (1,n!,n!,\ldots,n!)
\]
satisfies all of them. 
\subsection{The optimal size of one die}\label{individual-section}

The arithmetic lower bound forces almost all dice to be exponentially
large, but it leaves room for a prescribed die to be much smaller.  The
following theorem determines the optimal order of its size.

\begin{theorem}[The size of one die]\label{thm:individual-size}
Every die in a permutation-fair family of \(n\geq2\) dice has at least
\((n-1)^2/200\) faces.  Conversely, there is an absolute constant \(C>0\)
such that, for every \(n\geq2\) and every integer \(K\geq Cn^2\), there
is a permutation-fair family in which a specified die has exactly \(K\)
faces.  The remaining \(n-1\) dice can all have the same finite number of
faces.
\end{theorem}

Thus, if \(g(n)\) is the smallest possible number of faces of one
specified die in a permutation-fair family of \(n\) dice, then
\[
  g(n)=\Theta(n^2).
\]
The construction below gives no quantitative bound for the other dice,
and hence no improvement in the minimum total size.

\subsubsection{A universal quadratic lower bound}

Fix a die \(D_0\) with sorted faces \(x_1<\cdots<x_K\), and put
\(m=n-1\).  For each other die \(D_i\), define
\[
  p_i(q)=\Pr(D_i<x_q),\qquad 1\leq i\leq m,\quad 1\leq q\leq K.
\]
Each \(p_i(q)\) is nondecreasing in \(q\).  Independence, followed by
permutation fairness restricted to \(D_0\) and any subset of the other
dice, gives
\begin{equation}\label{individual-comparison-moments}
  \frac1K\sum_{q=1}^K\prod_{i\in S}p_i(q)=\frac1{|S|+1}
  \qquad(S\subseteq[m]).
\end{equation}
Indeed, the left side is the probability that \(D_0\) beats every die
indexed by \(S\).  We will use these identities with general positive
weights \(w_q\) summing to one in place of \(1/K\).  By linearity,
their weighted version implies
\begin{equation}\label{individual-diagonal-identity}
  \sum_{q=1}^K w_q F(p_1(q),\ldots,p_m(q))
  =\int_0^1 F(t,\ldots,t)\,dt
\end{equation}
for every multi-affine polynomial \(F\), meaning one whose degree in
each variable is at most one.  We write \(\mathbb E_w\) for the weighted
sum on the left.

\begin{lemma}[Paired thresholds]\label{individual-paired-thresholds}
Suppose the vectors \(p(q)\in[0,1]^m\) are coordinatewise
nondecreasing in \(q\).  For distinct indices \(j,k\) and any
\(t\in(0,1)\), there are thresholds \(\alpha,\beta\) such that
\[
  \{\alpha,\beta\}=\{t,b\},\qquad 0\leq b\leq t,
  \qquad
  (p_j(q)-\alpha)(p_k(q)-\beta)\geq0
  \quad\text{for every }q.
\]
\end{lemma}

\begin{proof}
Adjoin the all-zero and all-one vectors and interpolate successive
vectors linearly to form a continuous coordinatewise nondecreasing path.
Stop at the first point where either coordinate \(j\) or coordinate
\(k\) reaches \(t\), and let \(\alpha,\beta\) be their values there.
One equals \(t\), and the other is at most \(t\).  Every original row
before this point lies below both thresholds, and every row after it lies
above both.  Thus the two differences have the same weak sign.  The
interpolating points are used only to choose thresholds; they carry no
weight in~\eqref{individual-diagonal-identity}.
\end{proof}

We need the following elementary facts about quadrature for the uniform
measure on \([0,1]\).  Their proofs are included to specify the endpoint
estimates used below.

\begin{lemma}[Gaussian and Radau rules]\label{individual-quadrature-facts}
The \(D\)-node Gaussian rule has positive weights, is exact through
degree \(2D-1\), and has largest node \(\xi_D\) satisfying
\begin{equation}\label{individual-gaussian-edge}
  1-\xi_D\leq\frac{\pi^2}{4(D+1)^2}.
\end{equation}
The right Radau rule with interior nodes \(t_1<\cdots<t_e\) and endpoint
\(1\) has positive weights and is exact through degree \(2e\).  Its
endpoint weight and interior nodes satisfy
\begin{equation}\label{individual-radau-identities}
  \omega_{\mathrm{end}}=\frac1{(e+1)^2},
  \qquad
  \sum_{\ell=1}^e\frac1{1-t_\ell}=\frac{e(e+2)}2.
\end{equation}
\end{lemma}

\begin{proof}
Let \(P_r\) denote the Legendre polynomial normalized by \(P_r(1)=1\).
The Gaussian nodes are the shifted roots of \(P_D\).  Polynomial
division by the node polynomial and orthogonality give exactness through
degree \(2D-1\); applying the rule to squared cardinal polynomials gives
positive weights.  The unshifted nodes are the eigenvalues of the
symmetric \(D\)-by-\(D\) Jacobi matrix with zero diagonal and
off-diagonal entries
\[
  \frac{k}{\sqrt{(2k-1)(2k+1)}}\geq\frac12,
  \qquad 1\leq k<D.
\]
The matrix with all these entries replaced by \(1/2\) has largest
eigenvalue \(\cos(\pi/(D+1))\), with a positive sine eigenvector.
Using that vector in the Rayleigh quotient shows that the largest
unshifted Gaussian node is at least \(\cos(\pi/(D+1))\).  Rescaling to
\([0,1]\) and using \(1-\cos x\leq x^2/2\) proves
\eqref{individual-gaussian-edge}.

For the Radau rule, put
\[
  J_e(t)=\sum_{r=0}^e(2r+1)P_r(2t-1).
\]
Legendre orthogonality gives the reproducing identity
\(\int_0^1J_e(t)f(t)\,dt=f(1)\) for \(\deg f\leq e\).
Consequently \(J_e\) is orthogonal to all polynomials of degree at most
\(e-1\) for the weight \(1-t\), and has \(e\) distinct roots in
\((0,1)\).  Interpolation at these roots and at \(1\) gives a rule
exact through degree \(2e\): division by \((t-1)J_e(t)\) leaves a
remainder of degree at most \(e\), and the quotient has degree at most
\(e-1\).  The integral of the quotient term vanishes by orthogonality.
Applying the rule to
\((1-t)(J_e(t)/(t-t_\ell))^2\) shows that every interior weight is
positive.  At the endpoint, exactness on \(J_e^2\) gives
\[
  \omega_{\mathrm{end}}J_e(1)^2
  =\int_0^1J_e(t)^2\,dt=J_e(1)=(e+1)^2.
\]
Finally, the Legendre endpoint derivative identity gives
\[
  \frac{J_e'(1)}{J_e(1)}
  =\frac{\sum_{r=0}^e(2r+1)r(r+1)}{(e+1)^2}
  =\frac{e(e+2)}2.
\]
The logarithmic derivative of \(J_e\) at \(1\) is
\(\sum_{\ell=1}^e(1-t_\ell)^{-1}\), proving the remaining identity.
\end{proof}

\begin{lemma}[The last row]\label{individual-last-row}
Suppose \(w_q>0\), \(\sum_qw_q=1\), the vectors \(p(q)\in[0,1]^m\)
are coordinatewise nondecreasing, and
\eqref{individual-diagonal-identity} holds.  Write \(a_i=p_i(K)\).
Then
\[
  a_i\geq1-\frac{\pi^2}{(m+1)^2}\quad(1\leq i\leq m),
  \qquad w_K\leq\frac{200}{m^2}.
\]
\end{lemma}

\begin{proof}
Fix \(i\), put \(d=\lfloor(m-1)/2\rfloor\), and choose \(d\)
disjoint pairs of coordinates other than \(i\).  Let
\(\xi_1<\cdots<\xi_D\) be the nodes of the Gaussian rule with
\(D=d+1\) nodes.  For pair \(\ell\), apply
Lemma~\ref{individual-paired-thresholds} at \(\xi_\ell\), omitting
the largest node \(\xi_D\).  Let \(T\) be the product of the resulting
paired factors.  It is multi-affine and nonnegative on every row, and
\[
  P(t):=T(t,\ldots,t)
  =\prod_{\ell=1}^d(t-\xi_\ell)(t-b_\ell),
  \qquad b_\ell\leq\xi_\ell<\xi_D.
\]
Since \(a_i-p_i(q)\geq0\) and coordinate \(i\) is absent from \(T\),
\eqref{individual-diagonal-identity} and Gaussian exactness yield
\[
  0\leq\mathbb E_w[(a_i-p_i)T(p)]
  =\int_0^1(a_i-t)P(t)\,dt
  =\omega_D(a_i-\xi_D)P(\xi_D).
\]
Here \(\omega_D>0\) and \(P(\xi_D)>0\), also when \(d=0\).
Therefore \(a_i\geq\xi_D\).  Since \(D+1\geq(m+1)/2\), we obtain
\[
  1-a_i\leq\delta:=\frac{\pi^2}{(m+1)^2}
  \qquad\text{for every }i.
\]

For \(m\geq10\), put \(e=\lfloor m/10\rfloor\) and choose \(e\)
disjoint pairs of coordinates.  Apply
Lemma~\ref{individual-paired-thresholds} at the interior nodes of the
right Radau rule, one node for each pair.  Again let \(T\) be the
product of the paired factors, and write
\[
  P(t)=T(t,\ldots,t)
  =\prod_{\ell=1}^e(t-t_\ell)(t-b_\ell),\qquad b_\ell\leq t_\ell.
\]
Its degree is \(2e\leq m\).  Since \(P\) vanishes at the interior
Radau nodes, exactness gives
\begin{equation}\label{individual-radau-test}
  \mathbb E_w T(p)=\frac{P(1)}{(e+1)^2}.
\end{equation}
Put \(S=\sum_\ell(1-t_\ell)^{-1}=e(e+2)/2\).  The inequalities
\(e\leq m/10\) and \(e+2\leq3m/10\) imply
\[
  \delta S\leq\frac{3\pi^2}{200}<\frac14.
\]
In particular \(\delta<1-t_\ell\) for every \(\ell\), so all paired
thresholds lie strictly below all \(a_i\).  Each paired factor at the
last row, divided by its value at the all-one vector, is at least
\((1-\delta/(1-t_\ell))^2\).  Using
\(\prod_j(1-z_j)\geq1-\sum_jz_j\) for \(0\leq z_j\leq1\), we get
\[
  \frac{T(a)}{P(1)}
  \geq\prod_{\ell=1}^e
       \left(1-\frac{\delta}{1-t_\ell}\right)^2
  \geq1-2\delta S>\frac12.
\]
All rows contribute nonnegatively to
\eqref{individual-radau-test}.  Hence
\(w_K\leq2/(e+1)^2\leq200/m^2\).  For \(m<10\), the same bound
follows directly from \(w_K\leq1\).
\end{proof}

\begin{proof}[Proof of the lower bound in Theorem~\ref{thm:individual-size}]
Apply Lemma~\ref{individual-last-row} to the comparison probabilities
of any chosen die, with \(w_q=1/K\) and \(m=n-1\).
Equation~\eqref{individual-comparison-moments} supplies the required
identities, and the conclusion is \(K\geq(n-1)^2/200\).
\end{proof}

\subsubsection{A construction with a prescribed number of faces}

The upper bound starts with a real equal-weight quadrature with \(K\)
nodes.  We approximate the nodes by rationals and correct the resulting
order probabilities by changing the distributions of the other dice.
The corrections have disjoint supports and vanishing moments, so their
effect on every order probability is exactly linear.  Finally, we replace
the corrected continuous distributions by finite uniform dice without
changing the number \(K\) of distinguished faces.

\begin{lemma}[A regular starting quadrature]\label{individual-real-rule}
There is an absolute constant \(C_0\) such that, for every \(m\geq1\)
and every integer \(K\geq C_0(m+1)^2\), there are distinct nodes
\(0<u_1<\cdots<u_K<1\) satisfying
\[
  \frac1K\sum_{q=1}^K p(u_q)=\int_0^1p(t)\,dt
  \qquad(\deg p\leq m).
\]
We may take \(C_0\) large enough that \(K\geq m^2\).
\end{lemma}

\begin{proof}
Theorem~1.1 of Gilboa and Peled~\cite{gilboa2017}, applied to the
constant weight, gives an equal-weight quadrature of degree \(d\) with
exactly \(K\) nodes for every integer \(K\geq C_1d^2\).  In their
notation, the theorem bounds the threshold \(\overline N_w(d)\) for
all sufficiently large cardinalities, as well as the minimum
cardinality \(N_w(d)\).  For the constant weight their local-mass
quantity is of order \(d^2\); affine rescaling gives the assertion on
\([0,1]\).

Apply this result at degree \(2m+2\), with the prescribed \(K\).
The rule has at least \(m+2\) distinct support points: otherwise the
square of its support-vanishing polynomial would have degree at most
\(2m+2\), zero quadrature sum, and positive integral.  At least \(m\)
support points therefore lie in \((0,1)\).  Choose one node at each
of \(m\) distinct interior points as a pivot.  The Jacobian of the first
\(m\) moment sums with respect to these nodes has entries
\(s u_q^{s-1}/K\), \(1\leq s\leq m\), and is an invertible scaled
Vandermonde matrix.  The analytic implicit function theorem allows
arbitrary sufficiently small perturbations of all other nodes, with
compensating changes of the pivots preserving these \(m\) moments.

The pivots remain distinct and interior.  Move each free endpoint node
slightly into \((0,1)\).  Collisions between free nodes form proper
hyperplanes in the free parameters.  If a free node initially coincides
with a pivot, varying that free node alone changes the matching pivot
with derivative \(-1\), because their moment-Jacobian columns are
equal.  The difference of the two coordinates thus has nonzero derivative
at the initial point, so their collision set is a proper analytic zero
set.  Initially distinct coordinates remain distinct in a sufficiently
small neighborhood.  A finite union of these proper zero sets cannot
cover the open set of inward perturbations.  Choose a collision-free
perturbation and sort the nodes.  All \(K\) nodes are now distinct and
interior, and the moments through degree \(m\) are unchanged.
\end{proof}

\begin{lemma}[Local moment corrections]\label{individual-centered-bumps}
For distinct interior nodes \(u_q\), choose pairwise disjoint open
intervals \(I_q\subset(0,1)\) around them.  After shrinking these
intervals if necessary, there are neighborhoods \(V_q\subset I_q\)
of \(u_q\) and bounded piecewise-constant functions \(h_{q,t}\),
\(t\in V_q\), supported in \(I_q\), such that
\begin{equation}\label{individual-bump-moments}
  \int_0^t h_{q,t}(x)g(x)\,dx=g(t),\qquad
  \int_t^1 h_{q,t}(x)g(x)\,dx=-g(t)
  \qquad(\deg g\leq m-1).
\end{equation}
Their breakpoints are rational, and their values are rational when
\(t\) is rational.  Their coefficient vectors depend continuously on
\(t\), and the functions are uniformly bounded for \(t\) near \(u_q\).
In particular, their full moments through degree \(m-1\) vanish.
\end{lemma}

\begin{proof}
Inside each \(I_q\), choose \(m\) disjoint rational-endpoint intervals
strictly to the left of \(u_q\), and another \(m\) strictly to its
right.  Choose \(V_q\) between these two groups.  For any \(m\)
ordered disjoint intervals \(J_1<\cdots<J_m\) of positive length,
the matrix
\[
  M_{ra}=\int_{J_a}x^r\,dx,
  \qquad 0\leq r\leq m-1,\quad 1\leq a\leq m,
\]
is rational and invertible.  Indeed, multilinearity of the determinant
gives
\[
  \det M=\int_{J_1\times\cdots\times J_m}
         \prod_{a<b}(x_b-x_a)\,dx_1\cdots dx_m>0.
\]
Solve the left interval system with right-hand side
\((1,t,\ldots,t^{m-1})^T\), and the right interval system with its
negative.  Use the solutions as the values of \(h_{q,t}\) on those
intervals, and set it to zero elsewhere.  The equations give
\eqref{individual-bump-moments} on the monomial basis.  The matrices
are fixed, so their solutions have all the stated rationality,
continuity, and boundedness properties.
\end{proof}

The next lemma explains why these local corrections can repair every
full ordering simultaneously.  Write \(\mathcal P_r\) for the real
polynomials of degree at most \(r\).

\begin{lemma}[The response of order probabilities]\label{individual-order-response}
Choose \(t_q\in V_q\), pairwise disjoint sets
\(Q_1,\ldots,Q_m\subseteq[K]\), and coefficients \(\delta_{j,q}\)
such that
\[
  f_j(x)=1+\sum_{q\in Q_j}\delta_{j,q}h_{q,t_q}(x),\qquad j\in[m],
\]
are nonnegative densities.  Let \(X_0,\ldots,X_m\) be independent,
with \(X_0\) uniform on \(t_1,\ldots,t_K\) and \(X_j\) having density
\(f_j\) for \(j\in[m]\).  Define linear functionals on
\(\mathcal P_{m-1}\) by
\[
  B_j(a)=\frac1K\sum_{q\in Q_j}\delta_{j,q}a(t_q).
\]
If all \(B_j\) equal the same functional \(B\), then every specified
full ordering in which \(X_0\) has exactly \(k\) predecessors has
probability
\begin{equation}\label{individual-response-formula}
  \frac1K\sum_{q=1}^K p_k(t_q)+B(p_k'),
  \qquad
  p_k(t)=\frac{t^k(1-t)^{m-k}}{k!(m-k)!}.
\end{equation}
\end{lemma}

\begin{proof}
Fix the full ordering and condition on \(X_0=t_\ell\).  Expand the
product of the densities of the other variables.  Choosing the constant
term from every density gives probability \(p_k(t_\ell)\).

Consider a term choosing bumps for \(b\geq1\) variables, and expand
their sums over the bump indices.  These indices are distinct because
the sets \(Q_j\) are disjoint.  The chosen supports are therefore
disjoint and ordered.  If their order is inconsistent with the specified
full ordering, the term is zero.  Otherwise, integrate out the variables
whose selected density contribution was the constant \(1\).  The
resulting kernel is a product of powers of gaps between successive
chosen variables, \(t_\ell\), and the endpoints \(0,1\), divided by
factorials.  On the relevant ordered region it is a polynomial of total
degree \(m-b\).

Only the support \(I_\ell\) can contain \(t_\ell\).  If \(b\geq2\),
at least one chosen bump has its whole support on one side of
\(t_\ell\).  Because the support intervals are disjoint and ordered,
varying this coordinate within its support crosses neither another
chosen coordinate nor \(t_\ell\).  Hence the same gap-polynomial formula
holds throughout that support.  For fixed values of the other chosen variables, its
integration is over its entire support against a polynomial of degree
at most \(m-b\leq m-1\).  Its vanishing full moments make the term
zero.  The same argument eliminates a one-bump term unless its index
is \(\ell\).  Thus only one-bump terms cut by \(t_\ell\) can survive.

For such a bump from a variable before \(X_0\), let \(a\) variables
precede that variable and \(b\) lie between it and \(X_0\), so
\(a+b=k-1\).  Its kernel, with \(t=t_\ell\), is
\[
  \frac{x^a(t-x)^b(1-t)^{m-k}}{a!b!(m-k)!}.
\]
The left identity in~\eqref{individual-bump-moments} evaluates this
polynomial at \(x=t\), giving zero unless \(b=0\).  Hence only the
variable immediately preceding \(X_0\) contributes, with polynomial
\[
  A_k(t)=\frac{t^{k-1}(1-t)^{m-k}}{(k-1)!(m-k)!}
  \qquad(k\geq1).
\]
For a variable after \(X_0\), the kernel is
\[
  \frac{t^k(x-t)^a(1-x)^b}{k!a!b!},\qquad a+b=m-k-1.
\]
The right identity in~\eqref{individual-bump-moments} takes its
negative value at \(x=t\).  Only the variable immediately following
\(X_0\) contributes, with the negative of
\[
  C_k(t)=\frac{t^k(1-t)^{m-k-1}}{k!(m-k-1)!}
  \qquad(k\leq m-1).
\]
Define \(A_0=0\) and \(C_m=0\).  Both polynomials have degree at most
\(m-1\), and \(A_k-C_k=p_k'\).  Averaging over \(t_\ell\) applies
the functional \(B_j\) of the predecessor to \(A_k\) and that of the
successor to \(C_k\), omitting a term when the neighbor is absent.
Since all these functionals equal \(B\), the total correction is
\(B(p_k')\), as claimed.
\end{proof}

\begin{proof}[Proof of the construction in Theorem~\ref{thm:individual-size}]
Put \(m=n-1\) and fix the prescribed integer \(K\geq Cn^2\), where
\(C\) is large enough for Lemma~\ref{individual-real-rule}.  Choose
its nodes \(u_q\) and the corresponding neighborhoods and corrections
from Lemma~\ref{individual-centered-bumps}.  As \(K\geq m^2\), we can
choose disjoint sets \(Q_1,\ldots,Q_m\), each of size \(m\).

Approximate each \(u_q\) by a rational \(t_q\in V_q\), retaining their
strict order.  For \(1\leq s\leq m\), put
\[
  \mu_s(t)=\frac1K\sum_{q=1}^K t_q^s,
  \qquad a_s(t)=\frac{1/(s+1)-\mu_s(t)}s.
\]
For each \(j\in[m]\), solve the rational Vandermonde system
\begin{equation}\label{individual-correction-system}
  \frac1K\sum_{q\in Q_j}\delta_{j,q}t_q^{s-1}=a_s(t)
  \qquad(1\leq s\leq m).
\end{equation}
The distinct nodes make its matrix invertible.  Its solution depends
continuously on all the nodes near \(u\) and tends to zero as \(t\to u\),
since the starting quadrature gives \(a_s(u)=0\).  The bumps are
uniformly bounded near these starting nodes.  A sufficiently close
rational approximation therefore ensures that
\[
  f_j(x)=1+\sum_{q\in Q_j}\delta_{j,q}h_{q,t_q}(x)\geq\frac12
  \qquad\text{for every }x\in[0,1].
\]
The full zeroth moments of the bumps vanish, so each \(f_j\) integrates
to one.  These are positive piecewise-constant densities with rational
breakpoints and rational values.

Equation~\eqref{individual-correction-system} makes all functionals
\(B_j\) in Lemma~\ref{individual-order-response} equal to one
functional \(B\).  Checking on monomials and constants shows that,
for every \(p\in\mathcal P_m\),
\[
  B(p')=\int_0^1p(t)\,dt-\frac1K\sum_{q=1}^Kp(t_q).
\]
The lemma therefore gives every full ordering probability
\[
  \int_0^1\frac{t^k(1-t)^{m-k}}{k!(m-k)!}\,dt
  =\frac1{(m+1)!}.
\]
We have obtained a permutation-fair continuous model with one uniform
\(K\)-point variable.

It remains to realize this model by finite uniform dice.  Partition
\([0,1]\) at all breakpoints of the \(f_j\) and at every \(t_q\).
Ignore the boundary points, which have probability zero for the
continuous variables.  On each remaining interval \(J_g\), all the
densities are positive rational constants.  Thus the masses
\[
  \alpha_{j,g}=\int_{J_g}f_j(x)\,dx
\]
are positive rationals with \(\sum_g\alpha_{j,g}=1\) for every \(j\).
Choose a common denominator \(D\) for these masses.  By
Theorem~\ref{thm:first-size}, take an equal-sided permutation-fair
family on the \(m\) other dice, with \(F\) faces each.  For each
interval \(J_g\), make a copy of its fair word and replace each
occurrence of letter \(j\) by \(D\alpha_{j,g}\) consecutive copies.
Claim~\ref{clm:restriction-blowup}(ii) shows that the resulting block
is permutation-fair; it has \(FD\alpha_{j,g}\) faces of die \(j\).
Concatenate these blocks in interval order and insert one distinguished
face at each boundary \(t_q\).  Every other die now has exactly \(FD\)
faces, whereas the distinguished die still has exactly \(K\).

For each other die, the probabilities of selecting the different
intervals are exactly \(\alpha_{j,g}\), as in the continuous model.
Conditional on these interval choices, variables in different intervals
have a determined relative order.  Within a common interval, the
continuous variables are independent uniforms on that interval, so their
relative order is uniform.  The corresponding discrete variables also
have a uniform relative order, by fairness of the block and its
restrictions, using Claim~\ref{clm:restriction-blowup}(i).  Different
interval groups involve disjoint sets of independent dice, so their
internal orders combine independently.  Consequently the full order
distribution, including the distinguished faces at the interval
boundaries, is unchanged.  Assigning consecutive distinct integer labels
to the concatenated word produces the required finite dice.
\end{proof}

\section{Weaker notions of exact fairness}\label{sec:weaker}

For this section, an \emph{outcome} means one uniformly chosen occurrence of
every letter.  A string is \emph{go-first fair} if each letter is equally
likely to be the leftmost selected occurrence.  It is \emph{place fair} if,
for every rank, each letter is equally likely to occupy that rank.  These are
full-set notions: only the outcome involving all \(n\) dice is constrained.
We call either property \emph{hereditary} if it is required after restriction
to every nonempty subset of the dice.  In contrast, permutation fairness is
automatically hereditary by Claim~\ref{clm:restriction-blowup}(i).  The
full-set/hereditary distinction matters for both constructions and lower
bounds below.

\subsection{The exact full-set go-first optimum}

\begin{theorem}\label{thm:go-first-optimum}
Among full-set go-first-fair strings on \(n\) letters, allowing unequal
multiplicities, the minimum length is
\[
  L_{\mathrm{GF}}(n)=\frac{n(n+1)}2.
\]
\end{theorem}

\begin{proof}
For the upper bound, take
\begin{equation}\label{eq:go-first-staircase}
  s_n=12\cdots n\,(n-1)\,(n-2)^2\cdots 2^{n-2}1^{n-1}.
\end{equation}
Thus letter \(i\) occurs \(m_i=n-i+1\) times: once in the initial block and
\(n-i\) times in the descending staircase suffix.  For \(i<n\), letter \(i\)
wins precisely when it selects its initial occurrence and every letter
\(j<i\) selects one of its later occurrences.  Hence
\[
  \Pr(i\text{ wins})
  =\frac1{n-i+1}\prod_{j=1}^{i-1}\frac{n-j}{n-j+1}
  =\frac1n.
\]
Letter \(n\), which has a single occurrence, wins with probability
\[
  \prod_{j=1}^{n-1}\frac{n-j}{n-j+1}=\frac1n.
\]
The string is therefore go-first fair and has length \(n(n+1)/2\).

For the matching lower bound, let an arbitrary go-first-fair string have
letter multiplicities \(m_i\).  Order its letters as
\(i_1,\ldots,i_n\) according to the positions
\(p_1<\cdots<p_n\) of their first occurrences.  Let \(E_k\) be the event
that every earlier letter \(i_j\), \(j<k\), selects an occurrence after
\(p_k\).  If \(E_k\) fails, the winning selected occurrence belongs to one
of \(i_1,\ldots,i_{k-1}\).  Fairness therefore gives
\[
  \Pr(E_k^c)\leq\frac{k-1}{n},
  \qquad
  \Pr(E_k)\geq\frac{n-k+1}{n}.
\]
If \(i_k\) selects its first occurrence and \(E_k\) occurs, then \(i_k\)
wins.  These two events are independent, and consequently
\[
  \frac1n=\Pr(i_k\text{ wins})
  \geq\frac{\Pr(E_k)}{m_{i_k}}
  \geq\frac{n-k+1}{n m_{i_k}}.
\]
Thus \(m_{i_k}\geq n-k+1\).  Summing over \(k\) proves the lower bound.
\end{proof}

The hereditary versions of go-first and place fairness are arithmetically
more expensive even when the dice may have different numbers of faces.  Let
die \(i\) have \(m_i\) faces.  For every prime \(p\leq n\) and every
\(p\)-element set \(S\), hereditary go-first fairness forces
\(p\mid\prod_{i\in S}m_i\).  Hence at most \(p-1\) face counts can fail to
be divisible by \(p\), and therefore
\begin{equation}\label{eq:hereditary-primorial}
  \prod_{i=1}^n m_i
  \geq\prod_{\substack{p\leq n\\p\ \mathrm{prime}}}p^{\,n-p+1}
  =2^{\Theta(n^2)},
  \qquad \max_i m_i\geq2^{\Omega(n)}.
\end{equation}
The same argument at any fixed rank applies to hereditary place fairness.
When all \(m_i\) equal \(m\), the argument recovers
\(\operatorname{rad}(n!)\mid m\).

\subsection{An exact place-fair construction}

The Hahn construction already gives place-fair dice, as a consequence of
full permutation fairness, with \(2^{O(n\log n)}\) faces per die.  Thus the
construction below does not improve the asymptotic scale merely by weakening
the notion of fairness.  Its interest is instead structural: the rank counts
turn out to be controlled by equalities of moments through degree \(n-2\),
bringing the classical problem of equal sums of like powers directly into
the design of dice.  It gives a simple explicit construction independent of
the Hahn machinery, and every improvement in consecutive multigrade
partitions automatically yields place-fair dice with fewer faces.

Our route to exact place fairness is a consecutive
Prouhet--Tarry--Escott partition.  We first isolate the dice-specific
reduction, then retain Prouhet's especially transparent digit-sum
construction~\cite{prouhet1851}, and finally invoke the stronger partition
theorem of Choudhry~\cite{choudhry2020}.

Prouhet's construction and its relation to the later Tarry--Escott problem
are discussed by Lehmer and Wright~\cite{lehmer1947,wright1949,wright1959}.
Roberts studied partitions of consecutive integers and their digit
structure~\cite{roberts1964a,roberts1964b}, while Bolker et al. connect the
problem to generalized Thue--Morse sequences~\cite{bolker-et-al2016}.
For two classes, optimizing the interval length is also the problem of
constructing Littlewood polynomials with a high-order zero at one; see
\cite{boyd1997,boyd2002,berend-golan2006,buhler-et-al2021}.  These results
motivate shorter colorings, but the reduction to place-fair dice below works
for any number of colors.

\begin{lemma}[Prouhet digit-sum identity~\cite{prouhet1851}]\label{lem:prouhet}
Let \(n\geq2\), \(m=n^{n-1}\), and write
\(t=\sum_{\ell=0}^{n-2}d_\ell n^\ell\), where
\(0\leq d_\ell<n\).  Put
\[
  c(t)=\sum_{\ell=0}^{n-2}d_\ell\pmod n.
\]
For every \(0\leq q\leq n-2\), the sum
\[
  \sum_{\substack{0\leq t<m\\c(t)=a}}t^q
\]
is independent of \(a\in\mathbb Z/n\mathbb Z\).
\end{lemma}

\begin{proof}
Let \(\omega\) be a primitive \(n\)-th root of unity.  For
\(1\leq h<n\),
\[
  \sum_{t=0}^{m-1}\omega^{hc(t)}e^{tx}
  =\prod_{\ell=0}^{n-2}
    \sum_{d=0}^{n-1}(\omega^he^{n^\ell x})^d.
\]
Every factor on the right has a zero at \(x=0\), since \(\omega^h\neq1\).
Thus the product has a zero of order at least \(n-1\).  Taking derivatives
of orders \(q=0,1,\ldots,n-2\) at zero shows
\[
  \sum_{t=0}^{m-1}\omega^{hc(t)}t^q=0
  \qquad(0\leq q\leq n-2).
\]
Fourier inversion on \(\mathbb Z/n\mathbb Z\) now says exactly that the
power sums of all color classes are equal.
\end{proof}

\begin{theorem}[Moment-coloring reduction]\label{thm:place-moment}
Let \(n\geq2\) and \(m\geq1\) be integers, and suppose that there is a map
\[
  c:\{0,1,\ldots,m-1\}\longrightarrow\mathbb Z/n\mathbb Z
\]
such that, for every \(0\leq q\leq n-2\), the sum
\[
  \sum_{\substack{0\leq t<m\\c(t)=a}}t^q
\]
is independent of \(a\in\mathbb Z/n\mathbb Z\).  Then there are \(n\)
equal-sided, full-set place-fair dice with \(m\) faces each.
\end{theorem}

\begin{proof}
Fix a coloring \(c\) satisfying the hypothesis.  For \(0\leq t<m\), let \(B_t\) be
the cyclic block of length \(n\) whose letter at within-block position
\(r\in[n]\) is the unique \(i\in[n]\) satisfying
\[
  i\equiv r-c(t)\pmod n.
\]
Concatenate the blocks \(B_0B_1\cdots B_{m-1}\).  Every block contains every
letter once, so every die has \(m\) faces.

Fix a face of letter \(i\) in block \(t\), and suppose that its within-block
position is \(r\).  For each of the \(r-1\) letters before it in that block,
there are \(t+1\) faces before the chosen face and \(m-t-1\) after it.  For
each of the remaining \(n-r\) letters, the corresponding numbers are \(t\)
and \(m-t\).  Therefore the generating polynomial for the number of other
selected faces lying before this face is
\begin{equation}\label{eq:place-generating}
  F_{t,r}(z)=
  \bigl((m-t-1)+(t+1)z\bigr)^{r-1}
  \bigl((m-t)+tz\bigr)^{n-r}.
\end{equation}
For \(0\leq k<n\), put \(P_{r,k}(t)=[z^k]F_{t,r}(z)\).  This is a polynomial
in \(t\) of degree at most \(n-1\).  Its leading coefficient is
\[
  [z^k](z-1)^{n-1},
\]
which is independent of \(r\).

Letter \(i\) occupies position \(r\) in block \(t\) exactly when
\(c(t)\equiv r-i\pmod n\).  Hence its number of outcomes at rank \(k+1\) is
\begin{equation}\label{eq:place-count}
  C_{i,k}=\sum_{r=1}^n
     \sum_{\substack{0\leq t<m\\c(t)\equiv r-i\; (\mathrm{mod}\ n)}}
     P_{r,k}(t).
\end{equation}
For powers \(t^q\) with \(q\leq n-2\), the hypothesis makes the inner class
sum independent of its residue.  The coefficient of
\(t^{n-1}\) is independent of \(r\), while the classes
\(r-i\), \(r\in[n]\), run through all residues.  Thus~\eqref{eq:place-count}
is independent of \(i\).  For each outcome exactly one die occupies rank
\(k+1\), so every die occupies it in exactly \(m^n/n\) outcomes.  This holds
for every rank, proving place fairness.
\end{proof}

The digit-sum coloring of Lemma~\ref{lem:prouhet}, followed by
Theorem~\ref{thm:place-moment}, gives the especially transparent bound
\(m=n^{n-1}\).  The following stronger Prouhet--Tarry--Escott construction
improves this by a factor of \(n/2\).

\begin{theorem}[Improved place-fair upper bound]\label{thm:place-prouhet}
For every \(n\geq2\), there are \(n\) equal-sided, full-set place-fair dice
with
\[
  m=2n^{n-2}
\]
faces each.  For \(n=1\), one face suffices.
\end{theorem}

\begin{proof}
For \(n=2\), the digit-sum coloring of Lemma~\ref{lem:prouhet} has \(m=2\).
Now let \(n\geq3\).  Choudhry~\cite[Theorem~6]{choudhry2020} partitions the
consecutive integers
\[
  1,2,\ldots,2n^{n-2}
\]
into \(n\) equally sized classes whose sums of \(q\)-th powers agree for
every \(1\leq q\leq n-2\).  Equal class sizes give the assertion for
\(q=0\), and translating every integer by \(-1\) preserves equality of all
moments through degree \(n-2\).  The translated classes therefore satisfy
the hypothesis of Theorem~\ref{thm:place-moment} with \(m=2n^{n-2}\).
The same consecutive-interval partition theorem was independently obtained
by Pollack and Trevi\~no~\cite{pollack-trevino2026}.
\end{proof}

\section{Approximate permutation fairness}\label{sec:approximate}

Exact permutation fairness forces exponentially many faces on most dice.
We now ask how small the dice can be if every order is allowed a relative
error of at most $\varepsilon$.  Thus the required inequality is
\begin{equation}\label{eq:pointwise-approx}
  \left|n!P(\pi)-1\right|\leq\varepsilon
  \qquad\text{for every }\pi\in S_n.
\end{equation}
This controls even individually rare orders.  In particular, it implies
$\lVert P-U\rVert_1\leq\varepsilon$ and
$d_{\mathrm{TV}}(P,U)\leq\varepsilon/2$, where $U$ is uniform on $S_n$.
Here $\lVert P-U\rVert_1=\sum_{\pi\in S_n}|P(\pi)-1/n!|$ and
$d_{\mathrm{TV}}(P,U)=\frac12\lVert P-U\rVert_1$.

\subsection{Classical shuffle constructions}

The periodic word $(12\cdots n)^m$ realizes an inverse $m$-riffle shuffle:
each letter independently chooses a uniform level in $[m]$, and ties are
resolved in increasing letter order.  If $d=\operatorname{des}(\pi)$ is the
number of positions $j$ with $\pi_j>\pi_{j+1}$,
the classical $P$-partition count gives
\begin{equation}\label{eq:descent-count}
  P(\pi)=\frac1{m^n}\binom{m+n-1-d}{n},
  \qquad
  n!P(\pi)=\prod_{k=-d}^{n-1-d}\left(1+\frac{k}{m}\right)
  \quad(m\geq n).
\end{equation}
See Stanley~\cite{stanley1972} and Bayer and Diaconis~\cite{bayer-diaconis1992}.
The product decreases with $d$.  Its two extreme values show that
$m\geq n(n-1)/\varepsilon$ suffices for \eqref{eq:pointwise-approx} when
$0<\varepsilon\leq1$: the increasing-order ratio is at most
$e^{\varepsilon/2}\leq1+\varepsilon$, and the decreasing-order ratio is at
least $1-\varepsilon/2$.  Conversely, for $0<\varepsilon<1$ the decreasing
order forces $m\geq n$ and
\[
  m\geq\frac{n(n-1)}{2[-\log(1-\varepsilon)]}.
\]
Thus this periodic family requires $\Theta(n^2/\varepsilon)$ faces per
die for $0<\varepsilon\leq1/2$.

A better classical construction repeats a palindrome:
\[
  (12\cdots n\,n\cdots21)^m.
\]
Its order distribution is exactly the shelf shuffle of Diaconis, Fulman,
and Holmes~\cite{diaconis-fulman-holmes2013}, as given by their
Description~1.  Each die independently chooses one of the \(2m\)
monotone runs uniformly; reading the chosen letters groups them by run,
in increasing order within odd runs and decreasing order within even
runs.  This is precisely their packet construction, so their
Theorem~3.1 applies with the valleys of the output permutation itself.
Their Theorem~3.4 gives, for fixed $c>0$ and $m\sim cn^{3/2}$,
\[
  \max_{\pi\in S_n}|n!P(\pi)-1|
  \longrightarrow e^{1/(12c^2)}-1.
\]
Consequently, for each fixed positive error, $O(n^{3/2})$ faces per die
suffice.  Our construction improves this bound by independently relabeling
the palindrome blocks.

\subsection{Independent palindrome blocks}

For a permutation $\rho$, let $\rho^{\mathrm{rev}}$ be its reversal.
Choose $\rho_1,\ldots,\rho_m$ independently and uniformly from $S_n$, and
form the word
\begin{equation}\label{eq:random-palindrome-word}
  W=(\rho_1\rho_1^{\mathrm{rev}})
       \concat\cdots\concat(\rho_m\rho_m^{\mathrm{rev}}).
\end{equation}
Every resulting die has exactly $2m$ faces.  The randomness here chooses
the dice; after the word has been chosen, their order distribution $P_W$
is fixed.

\begin{theorem}[Approximate fairness from random palindromes]
\label{thm:approximate}
There is an absolute constant $C$ such that, for every $n\geq3$ and
$0<\varepsilon\leq1$, if
\[
  m\geq Cn^{13/10}(\log n)^{1/5}\varepsilon^{-2/5},
\]
then the random word \eqref{eq:random-palindrome-word} satisfies
\eqref{eq:pointwise-approx} simultaneously for all $\pi\in S_n$ with
probability at least $1/2$.  In particular, there exist $n$ equal-sided
dice with
\[
  O\!\left(n^{13/10}(\log n)^{1/5}\varepsilon^{-2/5}\right)
\]
faces each satisfying \eqref{eq:pointwise-approx}.
\end{theorem}

The palindrome structure makes every pair of letters exactly fair within
each block.  Thus a block contributes no error unless at least three dice
choose faces in it.  Independent relabeling centers the remaining errors;
the proof controls their joint contribution, including interactions between
several occupied blocks.

Here is the exact expansion behind that argument.  Fix a target order
$\sigma$.  Let $L=(L_1,\ldots,L_m)$ have the multinomial distribution with
$n$ trials and $m$ equally likely cells, and split $\sigma$ into consecutive
segments $A_i$ of lengths $L_i$.  For an ordered list $A$ of distinct
letters, write $q_i(A)$ for the probability that these letters appear in
the prescribed order when each chooses one of its two occurrences in
block $i$.  Put
\[
  \delta_i(A)=|A|!q_i(A)-1,
\]
with $\delta_i(\varnothing)=0$.  Summing over the chosen block indices gives
the identity
\begin{equation}\label{eq:palindrome-expansion}
  n!P_W(\sigma)=\mathbb E_L\prod_{i=1}^m(1+\delta_i(A_i)).
\end{equation}
For $|A|\leq2$, $\delta_i(A)=0$ identically.  For every fixed $A$,
uniform relabeling gives $\mathbb E_{\rho_i}\delta_i(A)=0$.

The following estimate controls the whole expansion.  Its proof, given in
Appendix~\ref{app:approximate-proof}, combines hypercontractivity with a
bound on multinomial collision probabilities.  The $L^p$ norm is over
the independent random permutations defining $W$.

\begin{lemma}[Uniform moment estimate]\label{lem:palindrome-moment}
Let $n\geq3$, $2\leq p\leq n^2$, and
$m\geq8n^{11/10}p^{1/5}$.  For every $\sigma\in S_n$,
\[
  \bigl\|n!P_W(\sigma)-1\bigr\|_p
  \leq32\,\frac{n^{11/4}\sqrt p}{m^{5/2}}.
\]
\end{lemma}

\begin{proof}[Proof of Theorem~\ref{thm:approximate}]
Set
\[
  p=\max\left\{2,\frac{\log(2n!)}{\log3}\right\}.
\]
For $n\geq3$, we have $p\leq n^2$ and $p=O(n\log n)$.
Choosing the absolute constant $C$ sufficiently large ensures
$m\geq8n^{11/10}p^{1/5}\varepsilon^{-2/5}$.
Lemma~\ref{lem:palindrome-moment} then gives, for every $\sigma$,
\[
  \bigl\|n!P_W(\sigma)-1\bigr\|_p
  \leq\frac{32}{8^{5/2}}\varepsilon<\frac\varepsilon3.
\]
Markov's inequality and a union bound imply
\[
  \Pr\!\left(\exists\sigma\in S_n:
      |n!P_W(\sigma)-1|>\varepsilon\right)
  \leq n!3^{-p}\leq\frac12.
\]
Thus a single sampled collection satisfies all the required inequalities
with the stated probability.
\end{proof}

\subsection{A lower bound for total size}

For fixed error, quadratic total length is necessary even if only the
first-place marginals are required to be approximately fair.  Suppose
each die is first with probability at most $(1+\varepsilon)/n$.
Order the letters as $i_1,\ldots,i_n$ according to the positions
$p_1<\cdots<p_n$ of their first occurrences.  Let $E_k$ be the event
that every earlier letter $i_j$, $j<k$, chooses an occurrence after
$p_k$.  If $E_k$ fails, the first selected occurrence belongs to one
of those $k-1$ letters.  Consequently,
\[
  \Pr(E_k)\geq1-\frac{(k-1)(1+\varepsilon)}n.
\]
The choice of a face of $i_k$ is independent of $E_k$.  If $i_k$
chooses its first occurrence and $E_k$ holds, then $i_k$ is first.
Writing $m_i$ for the face counts gives
\begin{equation}\label{eq:approx-first-lower}
  m_{i_k}\geq
  \max\left\{0,\frac{n}{1+\varepsilon}-(k-1)\right\}.
\end{equation}
Summing over $k$ yields total length $\Omega(n^2)$ for fixed
$\varepsilon$.  Condition~\eqref{eq:pointwise-approx} implies the marginal
hypothesis by summing over the orders in which a specified die is first.
Theorem~\ref{thm:approximate} gives total length
$O(n^{23/10}(\log n)^{1/5})$ for fixed positive error, leaving a factor
$O(n^{3/10}(\log n)^{1/5})$ between these upper and lower bounds.

\paragraph{Acknowledgments.}
We are especially grateful to Austin Conner, who devised the first version
of our second construction (Section~\ref{sec:second}) using polynomial
interpolation.
We thank Florian Haeupler for suggesting the problem, and James Grime,
Eric Harshbarger, Michael Purcell, Josef Tkadlec, and commenters on the
Polylog video~\cite{polylog-video} for helpful discussions.  T.G. acknowledges
the support of Charles University grant PRIMUS/22/HUM/020.
V.R. was supported by the Czech Science Foundation (GA\v{C}R), JUNIOR STAR
grant no.~26-23599M.

\paragraph{AI use.}
We collaborated with OpenAI's GPT models and Anthropic's Claude, which supplied
most of the proofs outside Section~\ref{sec:prelim}. The authors take responsibility for the
mathematical claims and the final text.

\clearpage
\appendix
\section{The moment estimate for random palindrome blocks}
\label{app:approximate-proof}

We prove Lemma~\ref{lem:palindrome-moment}.  Throughout this appendix,
$n\geq3$, $2\leq p\leq n^2$, and
$m\geq8n^{11/10}p^{1/5}$.  All norms and expectations over blocks refer
to the independent uniform permutations in
\eqref{eq:random-palindrome-word}, unless another random variable is
indicated explicitly.

\subsection{Two auxiliary estimates}

\begin{lemma}[Bounded multilinear chaos]\label{lem:palindrome-chaos}
Let $(Z_v)$ be independent centered random variables with $|Z_v|\leq b_v$.
If
\[
  T=\sum_{|A|=\ell}c_A\prod_{v\in A}Z_v,
\]
then, for $p\geq2$,
\[
  \|T\|_p\leq(2\sqrt{p-1})^\ell
     \left(\sum_{|A|=\ell}c_A^2\prod_{v\in A}b_v^2\right)^{1/2}.
\]
\end{lemma}

\begin{proof}
Introduce independent copies $Z_v'$ and replace every $Z_v$ in $T$ by
$Z_v-Z_v'$.  Conditional expectation over the copies recovers $T$, so
Jensen's inequality bounds the original norm by the new one.  Swapping
each pair $(Z_v,Z_v')$ independently preserves their joint law and changes
the corresponding difference by a Rademacher sign.  Conditional on the
pairs, the degree-$\ell$ Rademacher hypercontractive inequality
\cite{bonami1970,odonnell2014} bounds the $L^p$ norm over these signs by
$(p-1)^{\ell/2}$ times the coefficient $\ell_2$ norm.  Since
$|Z_v-Z_v'|\leq2b_v$, this gives the displayed bound uniformly in the
pairs, and hence also after averaging over them.
\end{proof}

For a probability vector $u=(u_0,\ldots,u_\ell)$, let
\[
  C_d(u)=\sum_{a_0+\cdots+a_\ell=d}
       \left(\frac{d!}{a_0!\cdots a_\ell!}
                    \prod_{j=0}^{\ell}u_j^{a_j}\right)^2
\]
be the collision probability of the multinomial distribution with $d$
trials and cell probabilities $u$.  We use the convention $0^0=1$.

\begin{lemma}[A lattice collision bound]\label{lem:palindrome-lattice}
For integers $\ell\geq1$, $d\geq0$, and $M\geq d+1$,
\[
  \sum_{g_0+\cdots+g_\ell=M}C_d(g/M)
  \leq\left(\frac{108M}{\sqrt{d+1}}\right)^\ell,
\]
where the sum is over nonnegative integer vectors $g$.
\end{lemma}

\begin{proof}
Elementary Stirling bounds give
\[
  \max_k\Pr(\operatorname{Poisson}(\lambda)=k)
     \leq\frac2{\sqrt{1+\lambda}},\qquad
  \Pr(\operatorname{Poisson}(d)=d)
     \geq\frac1{3\sqrt{d+1}}.
\]
The cases $\lambda\leq1$ and $d=0$ are immediate.  For $\lambda>1$,
the maximum is attained at $k=\lfloor\lambda\rfloor$, with
$k\geq\lambda/2$, and $k!\geq\sqrt{2\pi k}(k/e)^k$ gives the first
inequality.  For integer $d\geq1$,
$d!\leq e\sqrt d(d/e)^d$ gives the second.

Condition independent Poisson variables of means $du_j$ on their sum
being $d$.  Since a collision probability is at most the maximum
probability mass, the preceding bounds give
\[
  C_d(u)\leq3\,2^{\ell+1}\sqrt{d+1}
                   \prod_{j=0}^{\ell}(1+du_j)^{-1/2}.
\]
For every $g$ in the sum, some coordinate satisfies
$g_j\geq M/(\ell+1)$.  For that coordinate,
\[
  \sqrt{d+1}(1+dg_j/M)^{-1/2}\leq\sqrt{\ell+1}.
\]
Sum over the $\ell+1$ possible choices of this coordinate, then drop
the sum constraint on the remaining $\ell$ coordinates.  Each of them
lies between $0$ and $M$.  Since $w(x)=(1+dx/M)^{-1/2}$ is decreasing,
\[
  \sum_{g=0}^M w(g)
  \leq1+\int_0^M w(x)\,dx
  \leq1+\frac{2M}{\sqrt{d+1}}
  \leq\frac{3M}{\sqrt{d+1}}.
\]
The integral bound also holds for $d=0$; the last inequality uses
$M\geq d+1$.  Therefore the sum in the lemma is at most
\[
  6(\ell+1)^{3/2}
       \left(\frac{6M}{\sqrt{d+1}}\right)^\ell
  \leq\left(\frac{108M}{\sqrt{d+1}}\right)^\ell.
\]
For the final inequality, use $\ell+1\leq2^\ell$ and
$6\leq6^\ell$.
\end{proof}

\subsection{The expansion and a coloring argument}

Fix $\sigma\in S_n$ and write $R_\sigma=n!P_W(\sigma)$.  We use the
quantities $\delta_i(A)$ in \eqref{eq:palindrome-expansion}.  A prescribed
order of $k$ distinct letters occurs at most twice as a subsequence of
a palindrome $\rho_i\rho_i^{\mathrm{rev}}$.  Indeed, in the ordering
defined by $\rho_i$, the prescribed order must increase and then
decrease, and the only possible ambiguity is which occurrence of its
largest letter is chosen.  Hence, for $k\geq3$,
\begin{equation}\label{eq:palindrome-delta-bound}
  |\delta_i(A)|\leq b_k:=k!2^{1-k}.
\end{equation}
Here $q_i(A)\leq2^{1-k}$ and $b_k\geq1$.

Represent $\rho_i$ by the relative order of independent continuous
priorities $Y_{i,1},\ldots,Y_{i,n}$, indexed by positions in $\sigma$.
Write $\delta_{i,a,k}$ for the segment starting at position $a$ and
having length $k$.  For fixed $k$ and a fixed residue class of $a$
modulo $k$, these variables are independent: their segments use disjoint
sets of priorities.  Different blocks are independent as well.

Expanding \eqref{eq:palindrome-expansion}, the part involving exactly
$\ell$ nonconstant factors has the form
\[
  T_\ell=\sum_{I,\boldsymbol{k},\boldsymbol{a}}
        c(I,\boldsymbol{k},\boldsymbol{a})
                 \prod_{j=1}^{\ell}\delta_{i_j,a_j,k_j},
  \qquad I=\{i_1<\cdots<i_\ell\},\quad k_j\geq3.
\]
The coefficients are nonnegative probabilities of the corresponding
occupancies and segment starts.  The selected segments are disjoint and
appear in their natural order.  The sum of these homogeneous parts is
$R_\sigma-1$.

For any fixed restriction $\mathcal D$ on these indices, we claim that
\begin{equation}\label{eq:palindrome-colored}
  \|T_\ell^{\mathcal D}\|_p
  \leq(2\sqrt{p-1})^\ell
  \left[
    \sum_{(I,\boldsymbol{k},\boldsymbol{a})\in\mathcal D}
       c(I,\boldsymbol{k},\boldsymbol{a})^2
           \prod_{j=1}^{\ell}\frac{k_jb_{k_j}^2}{q_{k_j}}
  \right]^{1/2},
  \qquad q_k=2^{-(k-2)}\quad(k\geq3).
\end{equation}
Independently for each block, choose a size $K_i$ with probabilities
$q_k$, and a uniform residue $B_i$ modulo $K_i$.  Keep only terms with
$k_j=K_{i_j}$ and $a_j\equiv B_{i_j}\pmod{k_j}$, multiplying their
coefficients by $\prod_j k_j/q_{k_j}$.  Averaging over these auxiliary
choices recovers the original sum.  Given the choices, all remaining
variables are independent and centered, so
Lemma~\ref{lem:palindrome-chaos} applies.  Minkowski's inequality followed
by Jensen's inequality for the square root proves
\eqref{eq:palindrome-colored}: a term is retained with probability
$\prod_jq_{k_j}/k_j$.  A choice $K_i>n$ simply retains no terms from
block $i$.

We next compute the coefficients.  For fixed $I$ and
$\boldsymbol{k}$, put $s=\sum_jk_j$, $d=n-s$, and
\[
  g_0=i_1-1,\qquad
  g_j=i_{j+1}-i_j-1\ (1\leq j<\ell),\qquad
  g_\ell=m-i_\ell,\qquad M=m-\ell.
\]
The unspecified blocks lie in these $\ell+1$ gaps and
$\sum_jg_j=M$.  The segment starts are equivalently specified by the
numbers $a_0,\ldots,a_\ell$ of remaining letters assigned to those gaps,
where $\sum_ja_j=d$.  Direct summation of the multinomial law gives
\begin{equation}\label{eq:palindrome-coefficient}
  c(I,\boldsymbol{k},\boldsymbol{a})
  =\frac{\fall ns}{m^s\prod_jk_j!}
      \left(\frac Mm\right)^d
      \frac{d!}{a_0!\cdots a_\ell!}
                     \prod_{j=0}^{\ell}\left(\frac{g_j}{M}\right)^{a_j}.
\end{equation}
Here and below $\boldsymbol a$ denotes gap counts rather than the
equivalent segment starts.  The map $I\mapsto\boldsymbol g$ is a
bijection between $\ell$-subsets of $[m]$ and weak compositions of
$m-\ell$ into $\ell+1$ parts.  Since $m\geq n$ and
$\ell\leq n/3$, we always have $M>0$.

\subsection{Terms with short total support}

Split the expansion into terms with $s\leq n/2$ and $s>n/2$, denoting
their sums by $T^{\mathrm{short}}$ and $T^{\mathrm{long}}$.  In the
first case, $d\geq n/2$ and
$M=m-\ell\geq d+2\ell\geq d+1$.
Using \eqref{eq:palindrome-coefficient},
\eqref{eq:palindrome-delta-bound}, and $\fall ns\leq n^s$, the square
sum in \eqref{eq:palindrome-colored} for this part is at most
\[
  \sum_{\boldsymbol{k}:\,s\leq n/2}
       \prod_{j=1}^{\ell}
          \left[\left(\frac nm\right)^{2k_j}
                       2^{2-2k_j}\frac{k_j}{q_{k_j}}\right]
       \sum_{g_0+\cdots+g_\ell=M}C_d(g/M).
\]
We have also used $(M/m)^{2d}\leq1$.
Lemma~\ref{lem:palindrome-lattice} bounds the collision sum by
$(160m/\sqrt n)^\ell$, because $108\sqrt2<160$.
Extend the length sum to all $k_j\geq3$ to separate it into a product.
For $z=n^2/(2m^2)\leq1/2$,
\[
  \sum_{k\geq3}\left(\frac nm\right)^{2k}
                   2^{2-2k}\frac{k}{q_k}
  =\sum_{k\geq3}kz^k
  =\frac{z^3(3-2z)}{(1-z)^2}
  \leq8z^3=\frac{n^6}{m^6}.
\]
Consequently, if
\[
  a=2\sqrt{160}\,\frac{n^{11/4}\sqrt p}{m^{5/2}},
\]
then
\begin{equation}\label{eq:palindrome-short}
  \|T^{\mathrm{short}}\|_p
  \leq\sum_{\ell\geq1}a^\ell=\frac a{1-a}
  \qquad(a<1).
\end{equation}

\subsection{Terms with long total support}

For $s>n/2$, bound each collision probability by one and sum over $I$,
using $\binom m\ell\leq m^\ell/\ell!$.
For any $\theta\geq1$, the factor $\theta^{s-n/2}$ is at least one.
Inserting this factor into the nonnegative square sum in
\eqref{eq:palindrome-colored} gives
\[
  \|T_\ell^{\mathrm{long}}\|_p
  \leq\theta^{-n/4}\frac{A_\theta^\ell}{\sqrt{\ell!}},
  \qquad
  A_\theta=2\sqrt{pm}
       \left(\sum_{k\geq3}k
                  \left(\frac{\theta n^2}{2m^2}\right)^k\right)^{1/2}.
\]
Set $t=m/n\geq1$ and $\theta=t^{1/3}$.  The power-series argument
is $1/(2t^{5/3})\leq1/2$, and hence
\[
  A_\theta^2\leq4pm\,t^{-5}=\frac{4pn}{t^4}.
\]
Cauchy--Schwarz gives
\[
  \sum_{\ell\geq0}\frac{A^\ell}{\sqrt{\ell!}}
  \leq
  \left(\sum_{\ell\geq0}2^{-\ell}\right)^{1/2}
  \left(\sum_{\ell\geq0}\frac{2^\ell A^{2\ell}}{\ell!}\right)^{1/2}
  =\sqrt2e^{A^2}.
\]
Therefore
\begin{equation}\label{eq:palindrome-long}
  \|T^{\mathrm{long}}\|_p
  \leq\sqrt2\exp\left(-\frac n{12}\log t+\frac{4pn}{t^4}\right).
\end{equation}

\subsection{Completion of the moment bound}

Put
\[
  H=\frac{n^{11/4}\sqrt p}{m^{5/2}}.
\]
The assumption on $m$ gives $H\leq8^{-5/2}$, $a\leq1/6$, and
$t\geq8n^{1/10}p^{1/5}\geq8$.  Thus
\eqref{eq:palindrome-short} yields
\[
  \|T^{\mathrm{short}}\|_p
  \leq\frac65a=\frac{12}5\sqrt{160}\,H.
\]
Moreover, since $p\leq n^2$,
\[
  \frac{4p}{t^4}
  \leq\frac1{1024}\left(\frac p{n^2}\right)^{1/5}
  \leq\frac1{1024}\leq\frac{\log t}{24}.
\]
For $n\geq60$, \eqref{eq:palindrome-long} consequently gives
\[
  \|T^{\mathrm{long}}\|_p
  \leq\sqrt2\,t^{-n/24}
  \leq\sqrt2\,t^{-5/2}
  \leq H,
\]
where the last inequality uses $H=n^{1/4}\sqrt p\,t^{-5/2}$.
Combining the two parts gives
$\|R_\sigma-1\|_p\leq32H$ for $n\geq60$.

For completeness, the remaining $3\leq n<60$ are covered uniformly
by a simpler estimate that keeps each entire block as one variable.
Let $F_I$ be the part of the product expansion using exactly the
nonconstant factors indexed by $I$.  This is a function of those blocks,
centered separately in each block.  The selected-cell multinomial
marginal and \eqref{eq:palindrome-delta-bound} give
\[
  \|F_I\|_\infty\leq\beta^{|I|},\qquad
  \beta=\sum_{k\geq3}2^{1-k}(n/m)^k
       =\frac{n^3}{4m^3(1-n/(2m))}
       \leq\frac{n^3}{2m^3}.
\]
Indeed, the joint probability of selected occupancies $k_i$ with sum
$s$ is bounded by
\[
  \frac{n^s}{m^s\prod_i k_i!}.
\]
The factorials cancel against \eqref{eq:palindrome-delta-bound}, and
extending the sum to arbitrary $k_i\geq3$ gives the product above.

To bound the sum over $|I|=\ell$, introduce an independent copy of each
block.  Replace $F_I$ by its alternating difference in its $\ell$
arguments, denoted by $\Delta_I F_I$.  Separate centering implies
that averaging over the copies recovers $F_I$, while
$|\Delta_I F_I|\leq2^\ell\beta^\ell$.  Independently swapping each
block with its copy multiplies $\Delta_I F_I$ by the product of the
corresponding Rademacher signs.  Conditional hypercontractivity,
followed by the uniform bound on these coefficients, gives
\[
  \left\|\sum_{|I|=\ell}F_I\right\|_p
  \leq(2\sqrt{p-1})^\ell\binom m\ell^{1/2}\beta^\ell
  \leq\frac{b^\ell}{\sqrt{\ell!}},
  \qquad b=2\beta\sqrt{mp}.
\]
Since $b\leq n^{1/4}H\leq60^{1/4}8^{-5/2}<1/4$, we obtain
\[
  \|R_\sigma-1\|_p
  \leq\sum_{\ell\geq1}\frac{b^\ell}{\sqrt{\ell!}}
  \leq\frac b{1-b}
  \leq\frac43\,60^{1/4}H<4H.
\]
This proves Lemma~\ref{lem:palindrome-moment} for all $n\geq3$.

\begingroup
\sloppy
\bibliographystyle{unsrtnat}
\bibliography{ref}
\endgroup

\end{document}